\documentclass[final]{amsart}

\usepackage{graphicx}
\usepackage[section]{algorithm}
\usepackage{algorithmic}
\usepackage{enumerate}
\usepackage{mathrsfs,enumerate}
\usepackage{multirow,array}
\usepackage{graphicx,subfigure}
\usepackage{hyperref}
\usepackage{pstricks,pst-node,pst-tree, pst-plot, pst-eps}
\usepackage[notref,notcite]{showkeys} 
\usepackage{xspace}
\def\tph{\widetilde \pi_h}
\def\mbV{{\mathbb V}}

\def\tS{\widetilde S}
\def\Ats{(A^*)}

\def\dH#1{\dot{H}^{#1}}

\def\cQ{{\mathcal Q}}

\def\Forall{\qquad \hbox{for all }}

\def\beq#1\eeq{\begin{equation} #1 \end{equation}}
\def\bal#1\eal{\begin{aligned} #1 \end{aligned}}

\def\RR{{\mathbb R}}
\def\CC{{\mathbb C}}

\def\tH#1{\widetilde H^{#1}}

\chardef\atsign='100

\renewcommand{\Re}{\mathfrak{Re}}
\renewcommand{\Im}{\mathfrak{Im}}

\theoremstyle{plain}
\newtheorem{theorem}{Theorem}[section]
\newtheorem{lemma}[theorem]{Lemma}

\newtheorem{corollary}[theorem]{Corollary}

\theoremstyle{theorem}

\newtheorem{remark}{Remark}[section]

\newtheorem{assumption}{Assumption}

\newcommand{\step}[1]{\noindent\raisebox{1.5pt}[10pt][0pt]{\tiny\framebox{$#1$}}\xspace}

\graphicspath{{figures/}}

\newcommand{\modif}[1]{\textcolor{black}{#1}}

\begin{document}

\title[Numerical Approximation of Fractional Powers]
{Numerical Approximation of Fractional Powers of Regularly Accretive  Operators}

\author{Andrea Bonito}
\address{Department of Mathematics, Texas A\&M University, College Station,
TX~77843-3368.}
\email{bonito\atsign math.tamu.edu}

\author{Joseph E. Pasciak}
\address{Department of Mathematics, Texas A\&M University, College Station,
TX~77843-3368.}
\email{pasciak\atsign math.tamu.edu}

\date{\today}

\begin{abstract} 
We study the numerical approximation of fractional powers of accretive
operators in this paper.  Namely, if $A$ is the accretive operator
associated with a regular sesquilinear form $A(\cdot,\cdot)$ defined 
 on a Hilbert space 
$\mbV$ contained in  $L^2(\Omega)$, we approximate 
$A^{-\beta}$ for $\beta\in (0,1)$.   The fractional powers are defined 
in terms of the so-called Balakrishnan integral formula. 

Given a finite element approximation space $\mbV_h\subset \mbV$,
$A^{-\beta}$ is approximated by $A_h^{-\beta}\pi_h$ where $A_h$ is 
the operator associated with the form $A(\cdot,\cdot)$  restricted to
$\mbV_h$ and $\pi_h$ is the $L^2(\Omega)$-projection onto $\mbV_h$.
We first provide error estimates for $(A^{-\beta}-A_h^{-\beta}\pi_h)f$ in Sobolev
norms with index in [0,1] for appropriate $f$. These results depend on
elliptic regularity properties of variational solutions involving the 
form $A(\cdot,\cdot)$ and are valid for the case of less than full
elliptic regularity.   
We also construct and analyze an exponentially convergent sinc
quadrature approximation to the Balakrishnan integral defining 
$A_h^{-\beta}\pi_h f$.  Finally, the results
of numerical computations illustrating the proposed
method are given.
\end{abstract} 


\maketitle

\section{Introduction.}

The mathematical study of integral or nonlocal operators has received much
attention due to their wide range of applications, see for instance \cite{ISI:000175019600004,MR2450437,MR0521262,MR2223347,MR2480109,MR1918950,MR660727,MR1727557,MR1709781}.

Let $\Omega$ be a  bounded domain in $\RR^d$, $d\geq 1$, with a Lipschitz 
continuous boundary $\Gamma$ which is the disjoint union of an open set
$\Gamma_N$ and its complement $\Gamma_D=\Gamma\setminus \Gamma_N$.
We define $\mbV$ to be the
functions in $H^1(\Omega)$ vanishing on $\Gamma_D$.
We then consider a sesquilinear form
$A(\cdot,\cdot)$ defined for $u,v\in  \mbV$
given by
\begin{equation}\label{e:sesq_form}
\bal
A(u,v)&:= \int_\Omega\bigg( \sum_{i,j=1}^d a_{i,j}(x) u_{i}(x) \overline{v_j(x)} 
+\sum_{i=1}^d (a_{i,0}(x) u_{i}(x) \overline{v(x)} \\
&\qquad + a_{0,i}(x) u(x) \overline{v_i(x)})  +a_{0,0}(x) u(x)\overline{v(x)}\bigg).
\eal
\end{equation}
Here the subscript on $u$ and $v$ denotes the partial derivative with
respect to $x_i$, $i=1,..,d$, and $\overline{v}$ denotes the complex
conjugate of $v$.  
We further assume that $A(\cdot,\cdot)$ is coercive and bounded 
(see, \eqref{coer} and \eqref{bdd} below).  
Such a sesquilinear form is called \emph{regular} \cite{kato1961}.

There is an unbounded operator $A$ on $L^2(\Omega)$ with domain of
definition $D(A)$ associated with a regular sesquilinear form (see \cite{kato1961} and Section~\ref{s:notation}
below). The unbounded operator associated with such a form is called a
\emph{regularly accretive} operator \cite{kato1961}.  For such operators,
the fractional powers are well defined, typically, in terms of
Dunford-Taylor integrals.  When
$0<\beta<1$, one can
also use the Balakrishnan formula
\cite{balakrishnan,kato1960,kato1961}:
\beq 
A^{-\beta}=\frac {\sin(\beta \pi)}\pi \int_0^\infty \mu^{-\beta} (\mu I
+A)^{-1}\, d\mu.
\label{A-beta}
\eeq
In this paper, 
we propose a numerical method for the approximation $A^{-\beta}f$ based
on the finite element
method with an approximation space $\mbV_h\subset \mbV$.
It is worth mentioning that a similar formula can be used to represent the solution of a Cauchy problem.
This is the focus of \cite{wenyu}.

Several techniques for approximating $S^{-\beta}f$ 
are available when $S$ is the operator associated with a Hermitian 
form (or symmetric and real
valued on a real-valued functional space).
\modif{The most natural technique} involves approximating $S^{-\beta}$ by
$S_h^{-\beta}$  where $S_h$ is a discretization of $S_h$, e.g., via the
finite element method using $\mbV_h$.  In this case, $S_h^{-\beta}f$ for
$f\in \mbV_h$ can
be expressed in terms of the discrete eigenvector expansion \cite{MR2300467,MR2252038,MR2800568}:
$$S_h^{-\beta}f= \sum_j c_j \lambda_{j,h}^{-\beta}\psi_{j,h} \quad
\hbox{where}\quad f= \sum_j c_j \psi_{j,h}.$$
Here $(\lambda_{j,h},\psi_{j,h})$ denote the eigenpairs of $S_h$.
An alternative approach is based on a representation of $S^{-\beta}f$ via a ``Neumann to Dirichlet'' map \cite{MR2354493}.
\modif{The numerical algorithm proposed and analyzed in \cite{Abner, AbnerPost} consists
of a finite element method in one higher dimension.
It takes advantage of
the rapid decay of the solution in the additional direction enabling
truncation to a bounded domain of modest size.}
A third approach which is valid for more general $A$, 
is based on finite element approximation with an analysis
\modif{employing} the Dunford-Taylor characterization of
$A^{-\beta} f$ \cite{FujitaSuzuki} (see, also, \cite{MR1255054,Ushijima}) and is most closely related to the approach which we will
take in this paper.  However, the analysis of \cite{FujitaSuzuki}
 only provides errors in
$L^2(\Omega)$, requires full elliptic regularity   and fails
to elucidate
the relation between the convergence rate and the smoothness of $f$.
For example, the result in \cite{FujitaSuzuki} does not hold for problems
on non-convex domains or problems with jumps in the higher order 
coefficients.

The approach that we shall take in this paper is based on
\eqref{A-beta}.
The introduction of finite elements on a subspace $\mbV_h\subset \mbV$ 
leads to a discrete approximation $A_h$ to $A$.   The finite element
approximation to \eqref{A-beta} is then given by
\begin{equation}\label{A-betah}
A_h^{-\beta}\pi_h:=\frac {\sin(\beta \pi)}\pi \int_0^\infty \mu^{-\beta} (\mu I
+A_h)^{-1}\pi_h \, d\mu,
\end{equation}
where $\pi_h$
is the $L^2(\Omega)$-projection onto $\mathbb V_h$.
In \cite{bp-fractional}, we proved the convergence in $L^2(\Omega)$ of
an equivalent version of this method when $A=S$ is real, symmetric and
positive definite. We also showed  
the exponential convergence of a sinc quadrature approximation.

The current paper extends the approach of \cite{bp-fractional} 
to the case when $A$ is a regularly accretive operator.
The proof provided in \cite{bp-fractional} is based on the fact that, in
the Hermitian case, the domain of $S^\gamma$, for $\gamma\in \RR$, is naturally characterized
in terms the decay of the coefficients in expansions involving the
eigenvectors of $S$.  Assuming elliptic regularity, it is then possible
show that $D(S^{s/2})$ for $0\leq s \leq 1+\alpha$ coincides  with
standard Hilbert spaces.  Here $\alpha$ is the regularity parameter (see
below).  Thus, norms of the operator \modif{$(\mu I + S)^{-1}$} acting between the
standard Sobolev spaces can be bounded using their series expansions, 
the norms in $D(S^{s/2})$, and Young's inequalities.
In contrast, the spaces $D(A^\gamma)$ cannot be characterized in such a
simple way when $A$ is not Hermitian.

The main result of this paper is the following error estimate (Theorem~\ref{FEinterror} and Remark~\ref{r:conv_s12}): for any $0\leq r \leq 1$ there exists $\delta \geq 0$ and a constant $C$ independent of $h$ such that for $f\in D(A^\delta)$
\modif{\begin{equation}\label{e:error_intro}
\| (A^{-\beta}-A_h^{-\beta}\pi_h)f\|_{H^{r}(\Omega)} \leq C
h^ {\alpha+\min(1-r,\alpha)} 
\end{equation}
with $C$ being replaced by $C \log(h^{-1})$ for certain combinations of
$r$, $\alpha$, $\delta$  and $\beta$.  
Here }
$\alpha>0$ is the so-called elliptic regularity pick-up which is the
regularity above $H^1(\Omega)$ expected for $A^{-1}f$ for appropriate
$f$, 
see Assumption~\ref{ellreg}.  Even in the case of
Hermitian $A$, the above result extends those in \cite{bp-fractional} to $r>0$.

This paper shows that the general approach for proving \eqref{e:error_intro} 
in \cite{bp-fractional} can be extended to the case of regularly accretive
$A$ with additional technical machinery.    Some of the 
most challenging issues involve the relationship between 
$D(A^{s/2})$, for $s\in
[0,1+\alpha]$ and fractional Sobolev spaces.   
The case of $s\in [0,1)$ is contained in the acclaimed 
  paper by Kato \cite{kato1961} showing that for regularly accretive
  operators, $D(A^{s/2})$ coincides with the interpolation space
  between $L^2(\Omega)$ and $\mbV$ defined using the real method.
The case of $s=1$ is the celebrated Kato Square Root Problem.  This is a
deep result which has been intensively studied
(see, \cite{Agranovich,AxelKeithMc,mcintosh90,MR1255054} and the references in
\cite{mcintosh90}).  The results in those papers give conditions when
one can conclude that $D(A^{1/2})=D((A^*)^{1/2})=\mbV$, \modif{where  $A^*$ stands for the adjoint of $A$}.
Motivated by the approach of Agranovich and
Selitskii \cite{Agranovich} for proving the Kato Square Root Problem, we 
show in this paper that under elliptic regularity assumptions,
$H^s(\Omega)\cap \mbV \subset D(A^{s/2}) $ and $ H^s(\Omega)\cap \mbV
\subset D((A^*)^{s/2})$ for $s\in [0,1+\alpha]$
with equality when  additional injectivity assumptions on $A^{-1}$ and
$(A^*)^{-1}$ hold.  With this information, the norms of $(\mu I+A)^{-1}$
acting between Sobolev spaces can be bounded in terms of the $L^2(\Omega)$
operator norm of $A^t (\mu I+A)^{-1}$ with $t\in [0,1]$.  This, in turn,
 (see
Lemma~\ref{l:both_bounds})  
can be bounded by interpolation using the fact that $D(A^t)$ coincides
with the interpolation scale between $L^2(\Omega)$ and $D(A)$ (using 
the complex method).

Similar results for $0\leq s \le 1$  are also required  for the finite element approximation $A_h$.
In a discrete setting, the question is to guarantee the existence of a constant $C$ independent of $h$ such that for all $v_h \in \mathbb V_h$
$$
C^{-1} \| v_h \|_{H^{s}(\Omega)} \leq \| A_h^{s/2} v_h\|_{L^2(\Omega)} \leq C  \| v_h \|_{H^{s}(\Omega)}.
$$
This is provided by Lemma~\ref{l:characterization_discrete} for $0\leq s
< 1$ and 
a  solution to the discrete Kato problem ($s=1$) is given in Theorem~\ref{t:discrete_kato}.

We also study a sinc quadrature approximation to $A_h^{-\beta}f$ for
$f\in \mbV_h$.  A change of integration variable shows that 
\beq
A_h^{-\beta} =\frac {\sin(\beta \pi)}\pi \int_{-\infty}^\infty 
e^{(1-\beta)y} (e^y I
+A_h)^{-1}\, dy. \label{cint}
\eeq
Motivated by \cite{lundbowers}, the sinc quadrature approximation to 
$A_h^{-\beta}$ is given by
\beq
\cQ^{-\beta}_k (A_h):= \frac{k\sin(\pi \beta)}{\pi}      \sum_{\ell=-N}^N e^{(1-\beta) y_\ell} (e^{y_\ell}I+  A_h)^{-1}. 
\label{quadh}
\eeq
Here $k:=1/\sqrt N$ is the quadrature step size, $N$ is a positive
integer \modif{and $y_\ell:= \ell k$ for $\ell = -N,...,N$.}
The standard tools related to the sinc quadrature together with the
characterization of 
$D(A^{s/2})$ for $s\in [0,1/2]$ mentioned above 
yield the quadrature error estimate (Remark~\ref{rem:MN})
$$ 
\|(A_h^{-\beta}- \cQ^{-\beta}_k(A_h))\pi_h \|_{\dH{2s}\to \dH{2s}}
\le C_Q
e^{-\pi^2/(2k)},
$$
where $C_Q$ is a constant independent of $k$ and $h$.

The outline of this paper is as follows.
In Section~\ref{s:notation}, we introduce the notations and properties related to operator calculus with non-Hermitian operators.
Section~\ref{s:symmetric} is devoted to the study of the Hermitian part of $A$ and the related dotted spaces. 
The latter is instrumental for the characterization of $D(A^{\frac s 2})$ discussed in Section~\ref{s:characterization}.
The finite element approximations are then introduced in Section~\ref{s:fem}, which also contains the proof of the error estimate \eqref{e:error_intro}.
This coupled with the exponentially convergent sinc quadrature studied in Section~\ref{rem:MN}, yields the final error estimate for the fully discrete and implementable approximation. 
We end this work with Section~\ref{s:numerical} providing a numerical illustration of the approximation of fractional convection-diffusion problems.

\section{Fractional Powers of non-Hermitian Operators} \label{s:notation}

We recall that  $\Omega$ is a bounded domain in $\RR^d$ with a Lipschitz 
continuous boundary $\Gamma$ which is the disjoint union of an open set
$\Gamma_N$ and its complement $\Gamma_D=\Gamma\setminus \Gamma_N$.    
Let $L^2(\Omega)$ be the space of complex valued functions on $\Omega$ with square integrable absolute value and denote by $\|.\|$ and $(.,.)$ the corresponding norm and Hermitan inner product.
Let $H^1(\Omega)$ be the Sobolev space of complex valued  functions on
$\Omega$ and set
$$\mathbb V:=\{v\in H^1(\Omega),\ v=0 \hbox{ on } \Gamma_D\},$$
the restriction of $H^1(\Omega)$ to functions with vanishing traces on $\Gamma_D$.
We implicitly assume that $\Gamma_D$ is such that the trace operator
from $H^1(\Omega)$ is bounded into $L^2(\Gamma_D)$, e.g., 
$\Gamma_D$ does 
not contain any isolated sets of zero $d-2$ dimensional measure.
We denote  $\|.\|_{1}$ and $\|.\|_{\mathbb V}$ to be 
the norms on $H^1(\Omega)$ and $\mathbb V$ defined respectively by
\begin{equation*} 
\| v \|_{1}:= \left(\int_\Omega |v|^2+ \int_\Omega | \nabla v|^2\right)^{1/2}, \qquad 
\| v \|_{\mathbb V}:= \left( \int_\Omega | \nabla v|^2\right)^{1/2}.
\end{equation*}
For convenience, we avoid the situation where the variational space
requires $L^2(\Omega)$-orthogonalization, i.e., the Neumann problem 
without a zeroth order term.

For a bounded operator  $G:X(\Omega)\to Y(\Omega)$ between two Banach spaces $(X(\Omega),\|.\|_X)$ and $(Y(\Omega),\|.\|_Y)$ we write
$$
\| G \|_{X\to Y} := \sup_{u \in X, \| u \|_X =1} \| Gu\|_Y
$$
and, in short, $\| G \|:= \| G \|_{L^2 \to L^2}$.

Throughout this paper we use the notation $A \preceq B$ to denote $A\leq
C B$ with a constant $C$ independent of $A$, $B$ and the discretization 
mesh parameter
$h$ (defined later). 
When appropriate, we shall be more explicit on the dependence of $C$.

Consider the sesquilinear form \eqref{e:sesq_form}
for $ u,v$ in $\mathbb V$.
We assume that 
 $A(\cdot,\cdot)$ is strongly
elliptic  and
bounded.
The assumption
of strong ellipticity is the existence of a positive constant $c_0$
satisfying 
\beq
\Re(A(v,v))\ge c_0 \|v\|^2_{\mathbb V}, \Forall v\in \mathbb V.
\label{coer}
\eeq
The boundedness of $A(\cdot,\cdot)$ on $\mathbb V$ implies the existence of 
a positive  constant $c_1$ satisfying
\beq
|A(u,v)|\le  c_1 \|u\|_{\mathbb V}\|v\|_{\mathbb V}, \Forall u,v\in \mathbb V.
\label{bdd}
\eeq
The conditions \eqref{coer} and \eqref{bdd} imply that the sesquilinear form
$A(\cdot,\cdot)$  is regular on $\mbV$ (see, Section 2 of
\cite{kato1961}).

Following \cite{kato1961}, we define the Hermitian forms 
$$ \Re A(u,v) := \frac {A(u,v)+\overline{A(v,u)} } 2 \quad \hbox{ and }
\quad
\Im A(u,v) := \frac {A(u,v)-\overline{A(v,u)} }{2i}.$$
Note that \eqref{coer} implies that  $\Re A(u,u) $ is equivalent to 
$\|u\|_\mbV^2$, for all $v\in \mbV$ and
\eqref{bdd} is equivalent to 
\beq
|\Im A(u,u)| \le \eta\, \Re A(u,u), \Forall u\in \mbV,
\label{index}
\eeq
for some $\eta>0$.  The smallest constant $\eta$ above is called the
index of $A(\cdot,\cdot)$.

We now define operators associated with regular sesquilinear forms. 
Let $A$ be the unique closed m(aximal)-accretive operator 
of Theorem 2.1 of \cite{kato1961}, which is  defined as follows.
We set 
$\widetilde T:L^2(\Omega)\rightarrow \mbV$ by 
$$A(\widetilde Tu,\phi)=(u,\phi),\Forall \phi\in \mbV$$
(which is uniquely defined by the Lax-Milgram Theorem)
and set 
$$D(A):=\hbox{Range}(\widetilde T)\subset \mbV.$$  
As $\widetilde T$ is one to one, 
we define $Aw :=\widetilde T^{-1}w$, for $w\in D(A)$.
The operator $A$ associated with a regular sesquilinear form is said to be \emph{regularly accretive}.

\modif{We denote by $\mathbb V_a^*$ denotes the set of
bounded antilinear functionals on $\mathbb V$.}
It will be useful to consider also a related bounded operator 
$T_a: \mbV_a^* \rightarrow \mbV$ defined for $F\in
\mbV_a^*$ by the unique solution (Lax-Milgram again) to 
$$
A(T_aF,\phi)=\langle F,\phi \rangle,\Forall \phi \in \mbV.
$$
Here $\langle \cdot,\cdot\rangle $ denotes the antilinear
functional/function pairing.
The operator $T_a$ is a bijection of $\mbV_a^*$ onto $\mbV$ and we denote its inverse 
 by $A_a$.
>From the definition of $T_a$, we readily deduce that for $u\in \mbV$, $A_a u$ satisfies
\beq
\langle A_a u,v \rangle = A(u,v), \Forall v\in \mathbb V.
\label{Fa}
\eeq

It is also clear from the 
definition of $A$ that $u\in D(A)$ if and only if $A_au$ extends to a
bounded antilinear functional on $L^2(\Omega)$ and, then, 
$$(Au,v) =\langle A_au,v\rangle  = A(u,v),
\Forall u\in D(A),v\in \mathbb V.
$$

\modif{The above constructions can be repeated for adjoints defining
$\mathbb V_l^*$ the set of
bounded linear functionals on $\mathbb V$,} $\widetilde T^*:L^2(\Omega)\rightarrow \mbV$, $D(A^*):=\hbox{Range}(\widetilde T^*) \subset \mbV$,
$A^*:= (\widetilde T^*)^{-1}:D(A^*) \rightarrow L^2(\Omega)$, $T^*_l:\mbV_l^* \rightarrow \mbV$ and $A^*_l:=(T^*_l)^{-1}$.
In this case, for $v \in \mbV$, $A_lv$ satisfies
\beq
\langle  u,A_l v \rangle = A(u,v), \Forall u\in \mathbb V.
\label{Fa*}
\eeq
with $\langle \cdot,\cdot\rangle $ also denoting the function/linear
functional pairing.
We also have that $v\in D(A^*)$ if and only if $A_lv$ extends to a
bounded linear functional on $L^2(\Omega)$ and, then,
$$(u,A^*v) =\langle u,A_l^*v\rangle  = A(u,v),
\Forall u\in \mathbb V, v\in D(A^*).$$
Of course, these definitions imply that for  $u\in D(A)$ and $v\in D(A^*)$,
$(Au,v)=(u,A^*v)$.

By construction, the operators $A$ and $A^*$ defined from regular sesquilinear
forms are regularly accretive (cf. \cite{kato1961}). 
They satisfy
the following theorem (see, also \cite{FujitaSuzuki}):

\begin{theorem} [Theorem 2.2 of \cite{kato1961}] \label{resbound}
Let $A$ be the unique regularly accretive operator 
defined from a regular sesquilinear form
$A(\cdot,\cdot) $ with index $\eta$. Set $\omega:=\arctan(\eta)$.
Then the numerical range and the 
spectrum
of $A$ are  subsets of the sector $S_{\omega}:=\{z\in \CC\ : \ 
|\arg z|\leq\omega\}$.  Further, the resolvent set $\rho(A)$
of $A$ contains $S_{\omega}^c:=\CC \setminus S_{\omega}$ and on this set the resolvent 
$R_z(A):= (A-zI)^{-1}$ satisfies  
$$\|R_z(A)\| \le \left \{ 
\begin{array}{ll} 
[|z|\sin(\arg(z)-\omega)]^{-1} &\qquad
  \text{for} \quad \omega<\arg(z) \le \frac \pi 2 +\omega,\\ 
|z|^{-1} & \qquad \text{for} \quad \arg(z)>\frac \pi 2 +\omega.
\end{array} \right . 
$$
The result also holds for $A$ replaced by $A^*$.
\end{theorem}

\begin{remark} \label{cz2}
It easily follows from \eqref{coer} that for $\Re(z)\le c_0/2$,
$$\frac{c_0}2 \|u\|_1^2 \le |A(u,u)-z(u,u)|$$
which implies that 
$$\|R_z(A)f\|_1 \le \frac 2{c_0}\|f\|.$$
\end{remark}

It follows from Theorem~\ref{resbound} and Remark~\ref{cz2} that the Bochner 
integral appearing  in \eqref{A-beta}, for $\beta\in (0,1)$,
is well defined and gives a bounded operator $A^{-\beta}$
on
$L^2(\Omega)$.
Fractional powers for positive indices can be defined from those with 
negative indices.  
For $\beta\in
(0,1)$, 
$$D(A^\beta)=\{ u\in L^2(\Omega) \ : \ A^{\beta-1}u\in D(A)\}$$
and  $A^\beta u := A(A^{\beta-1})u$ for $u\in D(A^\beta)$.

An alternative but equivalent definition of fractional powers of 
positive operators (for $\beta\in (-1,1)$) is given in, e.g., 
\cite{lunardi}.  We shall recall some additional properties provided 
there (for $\beta\in(0,1)$).
Theorem 4.6 
of \cite{lunardi} implies that $D(A)\subset D(A^\beta)$ and 
for $v\in D(A)$, $A^\beta v = A^{\beta-1}
Av $ and $Av=A^\beta A^{1-\beta}v= A^{1-\beta} A^{\beta}v$.  Also, 
for any $\beta>0$,  $D(A^\beta) =  \{A^{-\beta}v\ : v\in L^2(\Omega)\}$ and
$A^\beta v= (A^{-\beta})^{-1}v$ for $v\in D(A^\beta)$.

Set $w:=A^\beta v$.  The last statement in the previous paragraph 
implies that $w = (A^{-\beta})^{-1}v$. Now as 
$ (A^{-\beta})w=v$,
$$\|v\| = \| A^{-\beta} w\| \le \| A^{-\beta} \| \|w\|
= \| A^{-\beta} \| \|A^\beta v\|.$$
This implies that we can take 
\begin{equation}\label{d:norm_DA}
\|v\|_{D(A^s)}:=\|A^s v\|
\end{equation}
 as our norm on $D(A^s)$ for $s\in [0,1)$.

Using the above and techniques from functional calculus \cite{haase},
we can conclude similar facts concerning products of fractional powers 
and the resolvent, namely,
\beq A^{-\beta} R_z(A) u = R_z(A) A^{-\beta} u, \Forall u\in
L^2(\Omega),\beta\ge 0,
\label{anegcom}
\eeq
and 
\beq A^{\beta} R_z(A) u = R_z(A) A^{\beta} u, \Forall u\in
D(A^\beta),\beta\in  [0,1].
\label{aposcom}
\eeq

We shall also connect fractional powers of 
operators with their adjoints in the $L^2(\Omega)$-inner product.
We have already noted that for $u\in D(A)$ and $v\in D(A^*)$, 
$(Au,v)=(u,A^*v)$.   This holds for fractional powers as well, i.e.,
$(A^{\beta} u,v)=(u,(A^*)^\beta v)$
provided that $u\in D(A)$ and $v\in D(A^*)$.

\section{The Hermitian Operator and the Dotted Spaces.}\label{s:symmetric}

For notational simplicity, we set $S(u,v):=\Re A(u,v)$.
As already noted, $S(u,u)^{1/2} $ provides an equivalent norm 
on $\mbV$ and we redefine $\|u\|_\mbV:=S(u,u)^{1/2}$.
As $S(\cdot,\cdot)$ is regular (i.e. satisfies \eqref{coer} and \eqref{bdd} with $A(.,.)$ replaced by $S(.,.)$), 
there is an
associated (m-accretive) unbounded operator $S$. The latter is defined similarly as $A$ from $A(.,.)$ in Section~\ref{s:notation} (see also \cite{kato1961}).
This is, upon first defining $T_S:L^2(\Omega)\rightarrow \mbV$ by 
$T_Sf:=w$ where $w\in \mbV$ is the unique solution of 
$$S(w,\phi)=(f,\phi) \Forall \phi\in \mbV$$
and then setting $D(S):=\hbox{Range}(T_S)$, $S:=T_S^{-1}$.

In addition, as $S(\cdot,\cdot)$ 
is symmetric and coercive, $S$ is self-adjoint and satisfies
$$S(u,v)=(S^{1/2}u,S^{1/2}v).$$

We  consider the Hilbert scale
of spaces defined by  $\dH s:=D(S^s)$ for $s\ge 0$.  
The above discussion implies 
\begin{equation}\label{e:dot_endpoints}
\dH1=\mbV \qquad \text{and} \qquad \dH0=L^2(\Omega).
\end{equation}
Moreover, the operator $T_S$ is a compact Hermitian operator on $L^2(\Omega)$
and so there is a countable $L^2(\Omega)$-orthonormal 
basis  $\{\psi_i, \ i=1,\ldots,\infty\} $ of
eigenfunctions for $T_S$.    The corresponding eigenvalues $\{\mu_i\}$ can
be ordered so that they are non-increasing with limit 0 and we set 
$\lambda_i=\mu_i^{-1}$.   This leads to a realization 
of $\dH s$ in terms of eigenfunction expansions, namely, for $s\in (0,1)$
\begin{equation}\label{e:dot_intermediate}
D(S^s):=\dH s =\bigg\{w=\sum_{j=1}^\infty (w,\psi_j) \psi_j 
\in L^2(\Omega) \ : \ \sum_{j=1}^\infty |(w,\psi_j)|^2
\lambda_j^s<\infty\bigg \}.
\end{equation}
The spaces $\dH s $ are Hilbert spaces with inner product
$$(u,v)_s := \sum_{j=1}^\infty \lambda_j^s (u,\psi_j) \overline
{(v,\psi_j)}.$$
Moreover, they are a Hilbert scale of spaces and are  also 
connected by the real
interpolation method.

As already mentioned $\dH 1 = \mathbb V$ so that the set of antilinear functionals on $\mathbb V$, denoted $\mbV^*_a$, can be characterized by 
$$
\mbV^*_a=\dH{-1}_a: =\bigg \{
\bigg \langle \sum_{j=1}^\infty c_j \psi_j,\cdot\bigg \rangle \ :  \
\sum_{j=1}^\infty |c_j|^2
\lambda_j^{-1}<\infty\bigg\},
$$
where $\bigg \langle \sum_{j=1}^\infty c_j \psi_j,  \sum_{i=1}^\infty d_j \psi_j  \bigg \rangle := \sum_{j=1}^\infty c_j \overline{d}_j$.

In addition, the set of antilinear functionals on $L^2(\Omega)$,
 denoted by
$L^2(\Omega)^*_a$, is given by 
$$L^2(\Omega)_a^*=\dH{0}_a: =\bigg \{\bigg\langle \sum_{j=1}^\infty  c_j
\psi_j,\cdot\bigg
\rangle\ : \ 
\sum_{j=1}^\infty |c_j|^2
<\infty\bigg\}.$$
Hence, the intermediate spaces are defined by
$$\dH{-s}_a: =\bigg \{
\bigg \langle \sum_{j=1}^\infty c_j \psi_j,\cdot\bigg \rangle \ :  \
\sum_{j=1}^\infty |c_j|^2
\lambda_j^{-s}<\infty\bigg\}$$
and are Hilbert spaces with the obvious inner product.  These also 
are a Hilbert scale of interpolation spaces for $s\in [-1,0]$.
In addition, these spaces are dual to $\dH s$,
i.e.,  if $s\in [0,1]$ and 
$$\langle w,\cdot \rangle = \bigg \langle \sum_{j=1}^\infty c_j \psi_j,\cdot\bigg \rangle \in \dH {-s}_a$$
then
\beq
\|w\|_{\dH {-s}_a}=\bigg(\sum_{j=1}^\infty \lambda_j^{-s} |c_j|^2
\bigg)^{1/2}
=\sup_{\theta\in \dH s} \frac {\langle w,\theta\rangle}
{\|\theta\|_{\dH s }}
\label{dual}
\eeq
and if $\theta\in {\dH s}$, 
\beq
\|\theta\|_{\dH {s}}
=\sup_{w\in \dH {-s}_a} \frac {\langle w,\theta\rangle}
{\|w\|_{\dH {-s}_a }}.
\label{dualother}
\eeq

Considering linear functionals instead of antilinear functionals 
and replacing $\mbV_a^*$ and $L^2(\Omega)^*_a$ by  
spaces of linear functionals, $\mbV_l^*$ and $L^2(\Omega)_l^*$ gives
rise to the analogous  Hilbert scale with $\dH{-1}_l=\mbV_l^*$ and 
$\dH0_l=L^2(\Omega)_l^*$ as endpoints with  equalities similar to 
\eqref{dual} and \eqref{dualother} holding for these as well.

As we shall see in Section~\ref{s:characterization}, $D(A^{s/2})$ relates either to $\dH s$ or $H^s(\Omega)\cap \mbV$
depending on whether $s>0$ is smaller or greater that $1$. 
In order to unify the presentation, we introduce the following spaces equipped with their natural
norms:
$$ \tH {s}:= \left \{\bal
\dH s &\qquad \hbox{for }s\in [0,1],\\
H^s(\Omega)\cap \mbV &\qquad \hbox{for }s\ge 1.
\eal \right.
$$

\section{Characterization of $D(A^{\frac s 2})$} \label{s:characterization}

In this section, we first observe that the dotted spaces for $s\in [0,1)$ 
coincide with the domains of 
fractional powers of $A$ and $A^*$ (c.f., \cite{kato1961}). In addition, 
we note that the dotted spaces $\dH{-s}_a$ and $\dH{-s}_l$ can be 
identified with the dual
space of  $\dH s$, for $s\in [0,1]$.  The case of $\tH s:=\dH s$, $s\in (0,1)$ 
is addressed in the following theorem which is an immediate consequence
of Theorem~3.1 of \cite{kato1961}.

\begin{theorem}[Characterization of $D(A^{\frac s 2})$ for $0\leq s <1$] \label{cont0_1_2}   
Assume that \eqref{coer} and  \eqref{bdd} hold.
  Then for $s\in [0,1 )$,
    $$
D(A^{s/2})=D((A^*)^{s/2})=D(S^{s/2})=\tH s,
$$
with equivalent norms.
\end{theorem}

The identification of the negative dotted spaces with with the duals 
is given in the following remark.

\begin{remark}[Characterization of Negative Spaces] \label{identify}
We identify $f\in L^2(\Omega)$ with the functional 
$F^f_a\in L^2(\Omega)^*_a$ defined by 
$$\langle F^f_a,\theta\rangle=(f,\theta),\Forall \theta\in L^2(\Omega).$$
It follows from Theorem~\ref{cont0_1_2} and \eqref{dual} that   the norms 
$\|F^f_a\|_{\dH {-2s} _a} $ and $\|A^{-s} f\|$ are equivalent (for $s\in
[0,1/2)$). Indeed,
$$\bal 
\|F_a^f\|_{\dH {-2s} _a}&=\sup_{\phi\in \dH {2s} } \frac {\langle F_a^f,\phi 
\rangle } {\|\phi\|_{\dH {2s}}}\approx \sup_{\phi\in \dH {2s} } \frac  {(f,\phi)} 
 {\|(A^*)^{s}\phi\|} \\
&=\sup_{\theta\in L^2(\Omega) } \frac {
(f,(A^*)^{-s}\theta)}
 {\|\theta\|_{\dH {2s}}} =\|A^{-s} f\|.
\eal
$$
Here $\approx$ denotes comparability with constants independent of $f$.
 For simplicity, we shall write 
$\|f\|_{\dH {-2s} _a}$ instead of $\|F^f_a\|_{\dH {-2s} _a} $.
We can identify $L^2(\Omega)$ with $L^2(\Omega)^*_l$ in an analogous way
and similar norm equivalences hold.
\end{remark}

Elliptic regularity is required to obtain convergence rates for finite
element approximation.   
Such results for boundary value problems have been
studied by many authors \cite{brbacuta,costabel,dauge,MR1173209,
  kellogg,nic1,Kondratev,nazplam,nicaise}. 
The follow assumption
illustrates the type of elliptic regularity  results available.
 
\begin{assumption}[Elliptic Regularity]\label{ellreg}
We shall assume elliptic
regularity for the form $A(\cdot,\cdot)$ 
with indices $\alpha \in (0,1]$. Specifically, we assume that for $s\in (0,\alpha]$,
$T_a$ is a bounded map of $\dot H_a^{-1+s}$ into $\tH{1+s}(\Omega)$
and $T_l^*$ is a bounded map of  $\dot H_l^{-1+s}$ into $\tH{1+s}$.
\end{assumption}
 
The above assumptions imply the following theorem.

\begin{theorem}[Property of $D(A^{\frac s 2})$ for $s > 1$]
 \label{t:characterization_hdot_onewayc}
Assume that \eqref{coer}, \eqref{bdd} and
 the elliptic regularity assumptions (Assumption~\ref{ellreg})
 hold.  Then for $s\in (1,1+\alpha]$,
$$
D(A^{s/2}) \subset \tH {s} 
\qquad \text{and} \qquad  D((A^*)^{ s/2}) \subset \tH s, 
$$
with continuous embeddings.
\end{theorem}

\begin{remark}[Kato Square Root Problem]  \label{re:KatoT}
The case of  $s=1$, i.e., 
$
D(A^{ 1/2}) \subset \mbV =: \tH 1 
$ with continuous imbedding
is contained in the Kato
Square Root Theorem.  This is a deep theorem which 
has been intensively studied, see  
\cite{Agranovich,AxelKeithMc,mcintosh90,MR1255054} and the references in
\cite{mcintosh90}.  The Kato Square Root Theorem  
holds for our problem under fairly weak regularity
assumptions on the coefficients defining our bilinear form
\cite{Agranovich}.
In fact, it requires the existence of $\epsilon>0$ such that $A_a$ and
$A_l^*$ are  bounded operators from $\tH{1+\gamma}$ to
$\dH{-1+\gamma}_a$ and $\dH{-1+\gamma}_l$, respectively,
for $| \gamma | \leq \epsilon$. 
\end{remark}

\begin{proof}[Proof of Theorem~\ref{t:characterization_hdot_onewayc}]
We consider the case of $A$ as the case of $A^*$ is similar.  
Suppose
that  $u$ is in $D(A)$ and $v$ is in $D(A^*)$.  
Then, for $t\in [0,\alpha]$,
$$ \bal A(u,v)&= (Au,v)
=(A^{(1-t)/2} A^{(1+t)/2} u,v)\\ &
= (A^{(1+t)/2}
u,\Ats^{(1-t)/2}  v):=F(v).
\eal
$$
Thus, Theorem~\ref{cont0_1_2} gives
$$\bal
|F(v)| 
&\le \|A^{(1+t)/2}u\|
\| \Ats^{(1-t)/2} v\| \\
& \preceq
\|A^{(1+t)/2}u\| \|v\|_{\dH{1-t}}.
\eal$$
This implies that $F\in \dot H_l^{-1+t}$.  
The elliptic regularity Assumption~\ref{ellreg} implies
that $u = T_l^* F$ is in $\tH{1+t}$ and satisfies
$$\|u\|_{H^{1+t}} \preceq \|A^{(1+t)/2}u\|.$$
As $D(A)$ is dense in $D(A^{(1+t)/2})$,
$D(A^{(1+t)/2})\subset \tH{1+t}$ 
follows.   
\end{proof}

\section{Finite element approximation}\label{s:fem}

In this section, we define finite element approximations 
to the  operator
$A^{-\beta}$
for $\beta\in (0,1)$.   
For simplicity, we assume that the domain $\Omega$ is polyhedral so that
it can be partitioned into a conforming subdivision made of simplices.
Further, we assume that we are given a finite dimensional subspace $\mathbb V_h\subset \mathbb V$
consisting of continuous complex valued functions, vanishing on
$\Gamma_D$,
which
are piecewise linear  with respect to a conforming subdivision  of
simplicies of maximal 
size diameter $h\le 1$.  
Notice that when the form $A(v,w)$ is real for real $v,w$, so is the finite element space (see Remark~\ref{r:real}).
We also need to assume that the triangulation
matches the partitioning $\Gamma=\Gamma_D\cup \Gamma_N$.  This means
that any mesh simplex of dimension less than $d$ which lies on $\Gamma$ 
is contained in either $\bar \Gamma_N$ or $\Gamma_D$.
Given a universal constant $\rho>0$, we restrict further our considerations to quasi-uniform partitions $\mathcal T$, i.e.  satisfying
\begin{equation}\label{e:quasiuniform}
\frac{\max_{T\in \mathcal T} \textrm{diam}(T)}{\min_{T\in \mathcal T} \textrm{diam}(T)} \leq \rho.
\end{equation}

Let $\pi_h$ denote the
$L^2(\Omega)$-orthogonal projector onto $\mathbb V_h$.
Given a sequence of conforming subdivisions $\{ \mathcal T\}_{h}$ satisfying \eqref{e:quasiuniform}, there holds
\beq
\|\pi_h v\|_1\le C \|v\|_1, \Forall v\in\mbV,
\label{pihb}
\eeq
where the constant $C$ is independent of $h$; see  \cite{bramblexu}. 
Obviously, $\pi_h$ is a bounded operator on $L^2(\Omega)$ and, by interpolation,
is a bounded operator on $\dH s$ for $s\in [0,1]$ with bounds
independent of $h$. 

We shall need the following lemma providing approximation properties for
$\pi_h$. \modif{ These results are a consequence of the Scott-Zhang approximation operator
\cite{ScottZhang} and operator interpolation.   As the specific form of
these results are needed for the analysis in the remainder of this paper,
we include a proof for completeness.}

\begin{lemma}  Let $s$ be in $[0,1]$ and $\sigma>0$ be such that
  $s+\sigma\le 2$.  Then there is a constant $C=C(s,\sigma)$ not
  depending on $h$ and satisfying
$$\|(I-\pi_h) u \|_{\tH s } \le C h^{\sigma} \|u\|_{\tH{s+\sigma}},
  \Forall u\in \tH{s+\sigma}.$$
\end{lemma}

\begin{proof}  
Let $\widetilde \pi_h$ denote the Scott-Zhang approximation operator
\cite{ScottZhang} mapping onto the set of piecewise linear polynomials
with respect to 
 the above triangulation (without any imposed boundary conditions).
This operator satisfies, for $\ell=0,1$ and $k=1,2$,
\beq
\|(I-\widetilde \pi_h) u \|_{H^\ell}\preceq h^{k-\ell} \| u \|
_{H^k}, \Forall u\in H^k(\Omega).
\label{scottz}
\eeq
In addition, $\widetilde \pi_hu \in \mbV_h$ for $u\in\mbV$.

We first verify the lemma when $s+\sigma\le 1$.  We clearly have
\beq
\| (I-\pi_h) u \| \le \| (I-\widetilde \pi_h) u \|\preceq h \| u \|
_{\mbV}, \Forall u\in \mbV
\label{sz0}.
\eeq
It immediately follows from \eqref{pihb}
that
\beq
\| (I-\pi_h) u \|_{\mbV} \preceq \| u \|
_{\mbV}, \Forall u\in \mbV.
\label{sz1}
\eeq
Interpolating this and \eqref{sz0} gives
\beq
\| (I-\pi_h) u \|_{\tH s} \preceq h^{1-s}\| u \|
_{\mbV}, \Forall u\in \mbV,
\label{sz2}
\eeq
for $s\in [0,1]$.  As $\pi_h$ is stable on $L^2(\Omega)$ and $\mbV$,
interpolation implies that it is stable on $\tH s$ and hence
\beq
\| (I-\pi_h) u \|_{\tH s} \preceq \| u \|
_{\tH s}, \Forall u\in \tH s,
\label{sz3}
\eeq
for $s\in [0,1]$. Interpolating \eqref{sz2} and \eqref{sz3} and applying
the reiteration theorem gives for $\sigma>0$ and $s+\sigma \le 1$,
\beq
\| (I-\pi_h) u \|_{\tH s} \preceq h^\sigma \| u \|
_{\tH {s+\sigma}}, \Forall u\in \tH {s+\sigma}.
\label{sz4}
\eeq

We next consider the case when $s+\sigma\in (1,2]$.  Taking $\ell=1$ in
\eqref{scottz} and interpolating between the $k=1$ and $k=2$ 
gives for $\zeta\in [1,2]$,
$$
\|(I-\widetilde \pi_h) u \|_{H^1}\preceq h^{\zeta-1} \| u \|
_{H^\zeta}, \Forall u\in H^\zeta(\Omega).
$$
Thus for $u\in \tH {s+\sigma}$, by \eqref{sz4},
$$\bal 
\| (I-\pi_h) u \|_{\tH s} &\preceq h^{1-s} \| (I-\pi_h) u \|
_{\mbV}\preceq h^{1-s} (\| (I-\widetilde \pi_h) u \|_\mbV 
+\| (\widetilde \pi_h -\pi_h) u\|_\mbV)\\
&\preceq  h^\sigma \|u\|_{\tH{s+\sigma}} +h^{-s} \|  (\widetilde \pi_h -\pi_h)
u\|
\preceq h^\sigma \|u\|_{\tH{s+\sigma}}
\eal
$$
where the last inequality followed from \eqref{scottz} and obvious
manipulations.   
\end{proof}

We define $A_h:\mathbb V_h\rightarrow \mathbb V_h$ by 
\beq
(A_h v_h,\varphi_h)=A(v_h,\varphi_h), \Forall \varphi_h\in \mathbb V_h.
\label{ah}
\eeq
The operator $A_h$ is the discrete analogue of $A$ and we analogously  define
$A_h^*$,  the discrete analogue of $A^*$. The fractional
powers $A_h^{-\beta}$ for $\beta>0$ are again given by
\eqref{A-beta} but with $A$ replaced by $A_h$, i.e.,  for $\beta\in(0,1)$,
$A_h^{-\beta}:\mathbb V_h \rightarrow \mathbb V_h$ is given by
\beq 
A^{-\beta}_h:=\frac {\sin(\beta \pi)}\pi \int_0^\infty \mu^{-\beta} (\mu I
+A_h)^{-1} \, d\mu.
\label{feA-beta}
\eeq

The goal of this paper is to analyze the error between $A^{-\beta} f$ 
and $A^{-\beta}_h \pi_hf$.

\begin{remark}[Real Valued Bilinear Forms and Finite Element Spaces]\label{r:real}  When $A(v,w)$ is real for real $v,w$,  the above 
operators restricted to real valued functions are real valued and
hence we may use Sobolev spaces  and approximation 
spaces $\mathbb V_h$ of real valued functions.  
\end{remark}

Similarly, let $S_h:\mbV_h\rightarrow\mbV_h$ be defined by
$$ (S_h v_h,w_h)=S(v_h,w_h),\Forall v_h,w_h\in \mbV_h.$$
Theorem 3.1 of \cite{kato1961} applied to the discrete operators  
$A_h$ and $S_h$ shows that for $s\in [0,1/2)$,
$$
\bigg(1-\tan\frac{\pi s}2\bigg )\|S_h^s v_h\| \le 
\|A_h^s v_h\| \le \bigg[1+\big(\frac s\pi \tan \pi s\big)^{1/2}\big
(\eta+\eta^2)\bigg] \|S_h^s v_h\|, 
$$
for all  $v_h\in \mbV_h$.  
Here  $\eta$ is the index of $A(\cdot,\cdot)$ (see, \eqref{index}).
This also holds for $A^*_h$.

The bound \eqref{pihb} 
implies (see, e.g., \cite{bankdupont}) that there are positive
constants $c$ and $C$, not depending on $h$ such that for $s\in [0,1]$,
\beq
  c\|S_h^{s/2} v_h\|\le \|v_h\|_{\dH s} \le C \|S_h^{s/2} v_h\|,
\Forall v_h\in \mbV_h.
\label{sh1}
\eeq
Combining the  preceding two sets of inequalities proves the following
lemma. 

\begin{lemma}[Discrete Characterization of $\dH {s} $ for $s\in \lbrack 0,1)$] \label{l:characterization_discrete}
There exist positive  constants $c,C$ independent of $h$ 
such that for all $v_h\in \mathbb V_h$ and $s\in [0,1)$,
$$ 
c \|v_h\|_{\dH s}\le \|A_h^{s/2} v_h\|\le C   \|v_h\|_{\dH s}.
$$
This result holds with $A_h$ replaced by $A_h^*$.
\end{lemma}

\begin{remark}[Quasi-uniformity Assumption] 
The above lemma still holds  as long as \eqref{pihb} holds. 
This is for instance the case for certain mesh refinement strategies \cite{bankyserentant,brPstein}.
\end{remark}

\section{Error Estimates}\label{s:error}

In this section, we study numerical approximation to the operators
$A^{-\beta}$ 
for $\beta\in (0,1)$. Specifically,
this involves bounding the errors $(A^{-\beta}-A_h^{-\beta}\pi_h)v$ 
for  $v$ having appropriate smoothness.

We shall use our finite element spaces to approximate $A_a^{-1}:=T_a$.
Specifically, 
$T_{h,a}:\mbV_a^* \rightarrow \mathbb V_h$ is defined for $F\in
\mbV_a^*$ by
$$
A(T_{h,a} F,\phi_h)=\langle F,\phi_h \rangle,\Forall \phi_h \in \mbV_h.
$$
We define $T_{h,l}^*$ corresponding to $T_l^*:=(A_l^*)^{-1}$
analogously.

\modif{
The following lemma provides approximation error bounds in terms of the norms
needed for our subsequent analysis.  Although the techniques in the
proof (Galerkin orthogonality and Nitsche finite element duality)  are
completely classical,  the results are not quotable (as far as we
know). 
We include a proof for completeness.}

\begin{lemma}[Finite Element Error]\label{FEerror} Assume that
  \eqref{coer}, \eqref{bdd} and the elliptic regularity  Assumption~\ref{ellreg} hold.
Let $s\in [0,\frac  12]$ and set $\alpha_*:=\frac 1 2 (\alpha+\min(1-2s,\alpha))$.
There is a positive constant $c$ not depending on $h$ 
satisfying
\begin{equation}\label{e:dual_estim}
\|T_a-T_{a,h}\|_{\dH {\alpha-1}_a \rightarrow \dH {2s}}
\le c h^{2\alpha_*}.
\end{equation}
The above immediately implies
$$\|T_a-T_{a,h}\|
\le c h^{2\alpha}.$$
\end{lemma}

\begin{proof} The proof of this lemma is classical and we only include
  details for completeness.  
We distinguish two cases.\\
\step{1} When  $2s \leq 1-\alpha$  then we can fully take advantage of the elliptic regularity assumption.
  For $F\in \dH {\alpha-1}_a$, 
we set $e=(T_a-T_{a,h})F$.  By \eqref{dualother} and the 
elliptic regularity
Assumption~\ref{ellreg},
$$\bal
\|e\|_{\dH {2s}} \leq  \|e\|_{\dH {1-\alpha}}&\preceq \sup_{G\in \dH {\alpha-1}_l}
\frac{\langle e,G\rangle} {\|G\|_{\dH {\alpha-1}_l}}\preceq
\sup_{G\in \dH {\alpha-1}_l}
\frac {A(e,T^*_l G)} {\|T^*_l G\|_{H^{1+\alpha}}}\\
&=\sup_{w\in \tH{1+\alpha}}
\frac{A(e,w)} {\|w\|_{H^{1+\alpha}}}=\sup_{w\in \tH{1+\alpha}}
\frac{A(e,w-w_h)} {\|w\|_{H^{1+\alpha}}}.
\eal
$$
We used Galerkin orthogonality for the last equality (which holds for
any $w_h\in \mbV_h$).
Using the approximation property 
$$\inf_{w_h\in \mbV_h} \|w-w_h\|_1 \preceq h^{\alpha}
\|w\|_{H^{1+\alpha}}, \Forall w\in \tH{1+\alpha},
$$
\eqref{bdd} and the above inequalities give
$$\|e\|_{\dH {1-\alpha}}\preceq h^{\alpha} \|e\|_1.$$

\step{2}This duality argument yields a reduced order of convergence when $2s>1-\alpha$.
Indeed, proceeding similarly
$$
 \bal
\|e\|_{\dH {2s}} &\preceq \sup_{G\in \dH {-2s}_l}
\frac{\langle e,G\rangle} {\|G\|_{\dH {-2s}_l}}\preceq
\sup_{G\in \dH {-2s}_l}
\frac {A(e,T^*_l G)} {\|T^*_l G\|_{H^{2-2s}}}\\
&=\sup_{w\in \tH{2-2s}}
\frac{A(e,w)} {\|w\|_{H^{2-2s}}}=\sup_{w\in \tH{2-2s}}
\frac{A(e,w-w_h)} {\|w\|_{H^{2-2s}}},
\eal
$$
so that together with the approximation property
$$\inf_{w_h\in \mbV_h} \|w-w_h\|_1 \preceq h^{1-2s}
\|w\|_{H^{2-2s}}, \Forall w\in \tH{2-2s},
$$
\modif{we deduce that}
$$
\|e\|_{\dH {2s}}\preceq h^{1-2s} \|e\|_1.
$$

Gathering the two cases $2s > 1+\alpha$ and $2s \leq 1+\alpha$, we get
$$
\|e\|_{\dH {2s}}\preceq h^{\min(1-2s,\alpha)} \|e\|_1.
$$
Whence, together with the estimate 
$$\|e\|_1\preceq h^\alpha\|T_aF\|_{H^{1+\alpha}}\preceq 
h^\alpha\|F\|_{\dH {\alpha-1}_a} $$
guaranteed by Cea's Lemma and elliptic regularity Assumption~\ref{ellreg},
we obtain \eqref{e:dual_estim} as desired.
\end{proof}

We can now state and prove our main convergence results. 
It requires data in the abstract space $D(A^\delta)$ for some $\delta  \geq 0$.
A characterization of $D(A^\delta)$ is provided in Theorem~\ref{dissqrrt} below. 

\begin{theorem}[Convergence] \label{FEinterror} 
  Suppose that  \eqref{coer} and
  \eqref{bdd} as well as
  the elliptic regularity Assumption~\ref{ellreg} hold. 
  Given $s\in [0,\frac 12)$, set $\alpha_*:=\frac 1 2 (\alpha+\min(1-2s,\alpha))$ and 
  $\gamma:=\max(s+\alpha^*-\beta,0)$ and let $\delta \geq \gamma$.
Assume finally, that $s+\alpha_*\neq 1/2$ when $\delta=\gamma$ and
$s+\alpha_*\ge \beta$.
Then there exists a constant $C$ independent of $h$ and $\delta$ such that
   $$
\| (A^{-\beta  } - A_h^{-\beta}\pi_h) f\|_{\dH{2s}} 
\leq C_{\delta,h}  h^{2\alpha_*}\|A^\delta f\|, \Forall f\in
D(A^{\delta}),
$$
where
\beq
C_{\delta,h}=\left \{ \bal C\ln(2/h)&:\qquad \hbox{ when } \delta
=\gamma \hbox{ and } s+\alpha_*\ge
 \beta, \ s+\alpha_* \not =\frac 1 2\\
C&:\qquad \hbox{ when } \delta>\gamma
\hbox{ and } s+\alpha_*\ge
\beta, \\
C&:\qquad \hbox{ when } 
\delta=0 \hbox{ and }\beta>s+\alpha_*.
\eal
\right.
\label{cdelta}
\eeq
\end{theorem}

\begin{remark}[Critical Case $s+\alpha_*=\frac 1 2$] \label{r:conv_s12} The condition $s+\alpha_* \not = \frac 1 2 $ in the above theorem can be removed provided that the Kato Square Root Theorem holds 
as well (see, Remark~\ref{re:KatoT}).
\end{remark}
\begin{remark}[Critical Case $2s=1$]\label{r:critical_s1}  The above results also hold when
  $2s=1$ provided that the continuous and discrete Kato Square Root
  Theorem hold. As already mentioned above, the former relies on the
  additional assumption requiring the existence of $\epsilon>0$ such
  that $A_a$ and $A_l^*$ are bounded from $\tH{1+\gamma}$
  to $\dH{\gamma-1}_a$ and $\dH{\gamma-1}_l$, respectively,
  for $|\gamma| \leq \epsilon$.
  For the discrete Kato Theorem, we will need to assume similar
  conditions for operators based on the $S$ form, see
  Theorem~\ref{dissqrrt} below.
  \end{remark}

The above theorem depends on an auxiliary lemma.

\begin{lemma}\label{l:both_bounds} 
For $s \in [0,1]$, there is a constant $C$ not depending on $h$ 
such that for any $\mu \in (0,\infty)$
$$
  \| A^{s}(\mu+A)^{-1}v\|\le C \mu^{s-1} \|v\|,\Forall
v\in L^2(\Omega)
$$
and
$$
 \|A_h^{s} (\mu+A_h)^{-1}v_h\|\le c \mu^{s-1}\|v_h\|,\Forall
v_h\in \mathbb V_h.
$$
\end{lemma}

\begin{proof}
The claim relies on interpolation estimates.
As the same argument is used for both estimates, we only prove the
first.

Theorem~4.29 of \cite{lunardi} implies $A^{it}$ is a bounded
operator satisfying 
\beq\|A^{it}\| \le e^{\pi |t|/2},\Forall t\in \RR.
\label{ait}
\eeq
This, in turn, implies that (e.g., Lemma~4.31 of \cite{lunardi}) for $s\in [0,\frac 1 2]$, 
\beq
[L^2(\Omega),D(A^{1/2})]_{2s} =
D(A^{s}).\label{complexi}
\eeq 
Here $[X,Y]_{2s}$ denotes the intermediate space between $X$ and $Y$
obtained by the complex interpolation method.
Thus, for $w \in D(A)$, Corollary~2.8 of \cite{lunardi} gives
$$
\| A^{s} w \| \preceq \| w \|_{[L^2(\Omega),D(A)]_{s}} \le \| A w\|^{s} \|w\|^{1-s}.
$$
Now if $w=(\mu+A)^{-1} v$ with $v\in L^2(\Omega)$, then 
$$\mu \|w\|^2 +\Re A(w,w)= \Re(v,w)$$
and hence \eqref{coer} immediately implies $\|(\mu+A)^{-1}v\|\le \mu^{-1} \|v\|$.
In addition,
$$\|A w\|= \|v-\mu(\mu+A)^{-1}v\| \le 2\|v\|.
$$
The lemma follows combining the above estimates.
\end{proof}

\begin{proof}[Proof of Theorem \ref{FEinterror}]
Without loss of generality, we may assume that $\delta\le 1+\alpha_*$
since we shall get $2\alpha_*$ order convergence as soon as
$\delta>2\alpha_*-2\beta$ and we always have $2\alpha_*-2\beta \leq 1+\alpha_*$.

\step{1}
We first show that 
\begin{equation}\label{e:first_step}
\|(I-\pi_h) A^{-\beta} f\|_{\dH{2s}}  \preceq h^{2\alpha_*}\|A^\delta f\|.
\end{equation}
Theorem~4.6 of \cite{lunardi} 
implies that $A^{-\beta} f$ is in $ D(A^{\alpha_*})$ when $f$ is in $D(A^{\alpha_*-\beta})$ and we now discuss separately the cases $s+\alpha_* \in (0,\frac 1 2)$, $s+\alpha_*=\frac 1 2$ and $s+\alpha_* \in (\frac 1 2, \frac 1 2 (1+\alpha)]$.
When $s+\alpha_* \in (0,\frac 12)$, we apply Theorem~\ref{cont0_1_2} 
and obtain
\begin{equation}\label{e:first_step_AA}
\begin{split}
\|(I-\pi_h) A^{-\beta} f\|_{\dH{2s}} &\preceq h^{2\alpha_*} \| A^{-\beta}
f\|_{\dH{2s+2\alpha_*}} \preceq h^{2\alpha_*} \| A^{s+\alpha_*-\beta}
f\|\\
&\preceq  h^{2\alpha_*} \| A^{\delta}
f\|,
\end{split}
\end{equation}
recalling that $\delta\geq \gamma \geq s+\alpha_*-\beta$.
For $s+\alpha_* \in (\frac 12, \frac 12(1+\alpha)]$, we apply
Theorem~\ref{t:characterization_hdot_onewayc} to conclude that
$A^{-\beta} f$ is in
$  \tH{2s+2\alpha_*}$. And again,
\begin{equation}\label{e:first_step_A}
\begin{split}
\|(I-\pi_h) A^{-\beta} f\|_{\dH{2s}} &\preceq h^{2\alpha_*} \| A^{-\beta}
f\|_{H^{2s+2\alpha_*}} \preceq  h^{2\alpha_*} \| A^{s+\alpha_*-\beta}
f\|\\
&\preceq  h^{2\alpha_*} \| A^{\delta}
f\|.
\end{split}
\end{equation}

Finally, we consider the case $s+\alpha_*=\frac 1 2$, which entails $\alpha \leq 1-s$ so that $\alpha_*=\alpha$ and $2s+2\alpha=1$.
We choose $0<\epsilon<\alpha$ (further restricted below) so that 
 as above 
$$
\|(I-\pi_h) A^{-\beta} f\|_{\dH{2s}} \preceq h^{2\alpha+\epsilon} \| A^{-\beta}
f\|_{H^{1+\epsilon}}\preceq  h^{2\alpha+\epsilon} \| A^{\frac{1+\epsilon}2 -\beta}
f\|.
$$
In addition the assumption $\delta > \gamma:=\max(\frac 1 2 - \beta,0)$ yields $\frac 12 -\beta + \frac \epsilon 2 <\delta$ upon choosing a sufficiently small $\epsilon$.
Hence, we deduce
$$
\|(I-\pi_h) A^{-\beta} f\|_{\dH{2s}}  \preceq h^{2\alpha} \| A^{\delta}f\|.
$$
This, \eqref{e:first_step_AA} and \eqref{e:first_step_A} yield \eqref{e:first_step}.

\step{2} By the triangle inequality,
it suffices now to bound
\begin{equation}\label{newi}
\begin{split}
&\|(\pi_h A^{-\beta} - A_h^{-\beta}\pi_h)\|_{D(A^\delta)\to\dH{2s}}\\
&\qquad \le\frac{\sin(\beta \pi)}{\pi} \int_0^\infty \mu^{-\beta}\left\|
\pi_h(\mu+ A)^{-1} - (\mu+A_h)^{-1}\pi_h \right\|_{D(A^\delta)\to \dH{2s}}\, d\mu .
\end{split}
\end{equation}
Assuming without loss of generality that $h\leq 1$, we shall break the above integral into integrals on three
subintervals, namely,  $(0,1)$, 
$(1,h^{-2\alpha_*/\beta})$ and $(h^{-2\alpha_*/\beta},\infty)$.

\step{3} We start with $(h^{-2\alpha_*/\beta},\infty)$.
Recalling the definition of the operator norm \eqref{d:norm_DA} as well as the characterizations of the dotted space provided by Theorem~\ref{cont0_1_2} , we get 
\begin{align*}
\| \pi_h (\mu+A)^{-1}\|_{D(A^{\delta})\to \dH{2s}}  \preceq 
\| A^{\max(s-\delta,0)} (\mu+A)^{-1}\|,
\end{align*}
where we used in addition the stability of $\pi_h$ in $D(A^\delta)$ and the boundedness of $A^{-r}$ from $L^2(\Omega)$ to $L^2(\Omega)$, for $r\geq 0$ (see discussion below Remark~\ref{cz2}).
Hence, applying Lemma~\ref{l:characterization_discrete} yields
$$
\| \pi_h (\mu+A)^{-1}\|_{D(A^{\delta})\to \dH{2s}} \preceq \mu^{\max(s-\delta,0)-1}.
$$
Similarly, but using the discrete characterization provided by Lemma~\ref{l:characterization_discrete}, we obtain
$$
\|  (\mu+A_h)^{-1}\pi_h\|_{D(A^{\delta})\to \dH{2s}} \preceq \mu^{\max(s-\delta,0)-1}.
$$
Thus, invoking Lemma~\ref{l:both_bounds}, we deduce that
\begin{align*}
& \int_{h^{-2\alpha_*/\beta}}^\infty  \mu^{-\beta} \|
\pi_h(\mu+ A)^{-1} - (\mu+A_h)^{-1}\pi_h \|_{D(A^\delta)\to \dH{2s}} \, d\mu \\
&\qquad  \preceq \int_{h^{-2\alpha_*/\beta}}^\infty  \mu^{-\beta+\max(s-\delta,0)-1}d\mu
\preceq \int_{h^{-2\alpha_*/\beta}}^\infty  \mu^{\max(-\alpha_*,-\beta)-1}d\mu
 \preceq h^{2\alpha},
\end{align*}
because $\delta \geq s+\alpha_*-\beta$.

\step{4} For the remaining two cases, we use the identity
$$
\pi_h(\mu+ A)^{-1} - (\mu+A_h)^{-1}\pi_h = (\mu+A_h)^{-1}A_h\pi_h(T_a-T_{h,a})A(\mu+A)^{-1}.
$$
The latter, follows from the identification of Remark~\ref{identify} and that the observation that 
for $u\in D(A)$ and $v\in \mbV$,
$$(T_aAu,v)=A(T_aAu,T_l^*v)=(Au,T_l^*v)=A(u,T_l^*v)=(u,v),$$
i.e., $T_aAu=u$.  Also, it is easy to see that 
$A_h \pi_h T_{h,a} =\pi_h$.
Thus,  for $u\in D(A)$,
$$A_h \pi_h (T_a-T_{h,a}) A u = (\mu +A_h   )\pi_h u - \pi_h (\mu +A ) u,$$  
which leads to the desired identity.

\step{5} For $\mu \in (1,h^{-2\alpha_*/\beta})$, we write
\begin{align*}
 \int_1^{h^{-2\alpha_*/\beta}} & \mu^{-\beta}\|
 A_h(\mu+A_h)^{-1}\pi_h(T_a-T_{h,a})A(\mu+A)^{-1}\|_{D(A^\delta)\to \dH{2s}} \, 
d\mu  \\
&\le  \int_1^{h^{-2\alpha_*/\beta}} \mu^{-\beta}  \|
A_h(\mu+A_h)^{-1}\pi_h\|_{\dH{1-\alpha_*} \rightarrow \dH{2s}}
\|(T_a-T_{h,a})\|_{\dH{-1+\alpha_*}_a\rightarrow \dH{1-\alpha_*}}\\
&\qquad \qquad \qquad
\|A(\mu+A)^{-1}\|_{D(A^\delta) \rightarrow \dH{-1+\alpha_*}_a}
 d\mu
\end{align*}

Now,  the definition \eqref{d:norm_DA} of $\|.\|_{D(A^\delta)}$ together with the characterization of the negative spaces provided in Remark~\ref{identify},
imply that
$$
\| A(\mu+A)^{-1}\|_{D(A^\delta) \to \dot H^{\alpha_*-1}_a} \preceq  \|A^{(1+\alpha_*)/{2}-\delta}(\mu+A)^{-1}\| \preceq \mu^ {(\alpha_*-1)/2-\delta}. 
$$
In addition,  using  Lemma \ref{l:characterization_discrete},
we obtain
\begin{align*}
 \| A_h(\mu+A_h)^{-1}\pi_h\|_{\dot H^{1-\alpha_*} \to  \dH{2s}} &=
 \sup_{f\in \dot H^{1-\alpha_*}} \frac{\|
   A_h(\mu+A_h)^{-1}\pi_h f\|_{\dH{2s}}}{\|\pi_h f\|_{\dot H^{1-\alpha_*}}}\\
 &\preceq  \sup_{f\in \dot H^{1-\alpha_*}}
 \frac{\| A_h(\mu+A_h)^{-1}\pi_h
   f\|_{\dH{2s}}}{\|A_h^{(1-\alpha_*)/2}\pi_h f\|}\\
&\preceq  \sup_{g\in \mathbb V_h} \frac{\|
  A_h^{\frac{1+\alpha_*}{2}+s}(\mu+A_h)^{-1}g\|}{\|g\|}\\
&\preceq  \mu^{\frac{-1+\alpha_*}{2}+s}.
\end{align*}

The above three estimates together with 
Lemma~\ref{FEerror} yield
\begin{align*}
\int_1^{h^{-2\alpha_*/\beta}}  &\mu^{-\beta}\| A_h(\mu+A_h)^{-1}\pi_h(T_a-T_{h,a})A(\mu+A)^{-1}\|_{D(A^\delta)\to \dH{2s}} \, d\mu  \\
&\preceq  h^{2\alpha_*} \int_1^{h^{-2\alpha/\beta}} \mu^{-1+s+\alpha_*-\beta-\delta} \preceq  \left \{ 
\begin{array}{ll}
h^{2\alpha_*}\ln(2/h)&:\qquad \hbox{ if } \delta=s+\alpha_*-\beta,\\
h^{2\alpha_*}&:\qquad \hbox{ otherwise.}
\end{array}
\right.
\end{align*}

\step{6} Finally for $\mu \in (0,1)$, we write
\begin{align*}
 &\int_0^1 \mu^{-\beta} \| A_h(\mu+A_h)^{-1}\pi_h(T_a-T_{h,a})A(\mu+A)^{-1}\|_{D(A^\delta)\to \dH {2s}} \, d\mu \\
 & \qquad \preceq \int_0^1 \mu^{-\beta}    \| A_h(\mu+A_h)^{-1}\pi_h\|_{\dH{2s}\to \dH{2s}} \| T_a-T_{h,a}\|_{L^2 \to \dH{2s}} \| A(\mu+A)^{-1}\|_{L^2 \to L^2}\, d\mu  \\
 & \qquad \preceq  h^{2\alpha} \int_0^1 \mu^{-\beta} \, d\mu \preceq   h^{2\alpha}.
\end{align*}

\step{7} \modif{The proof of the theorem is complete upon collecting the estimates obtained in Steps 3,5 and 6}.
\end{proof}

Theorem~\ref{cont0_1_2} and the Kato Square Root Theorem characterize
$D(A^s)$ for $s\in [0,1/2]$.  The characterization can be extended to
$s\in (1/2,(1+\alpha)/2]$ when $A_a$ maps $\tH{2s}$ into
$\dH{2s-2}_a$.
This is of particular importance to characterize the regularity assumption $f\in D(A^\delta)$ in Theorem~\ref{FEinterror}.

\begin{theorem}[Characterization of $D(A^{(1+s)/2})$ for $s \in (0,\alpha \rbrack$] \label{dissqrrt}
    Suppose that  \eqref{coer} and 
  \eqref{bdd} hold.
  Assume furthermore that for  $s \in (0,\alpha ]$,   
  
 \begin{equation}\label{a:Ta} 
 T_a \text{ is an isomorphism from }\dot H_a^{-1+s} \text{into }\tH{1+s}.
 \end{equation}
   Then,
$$
D(A^{(1+s)/2})=\tH{1+s}
$$
with equivalent norms.
\end{theorem}

\begin{proof}  By Theorem~\ref{t:characterization_hdot_onewayc},
  we need only prove that $\tH{1+s}\subset D(A^{
    (1+s)/2})$.
    We first observe that 
 $D(A)\cap
\tH{1+s}$  is dense in 
$\tH{1+s}$.  
Indeed, if $w$ is in $\tH{1+s}$ then \eqref{a:Ta} implies that $A_a w$
is in $\dot H_a^{s-1}$.  As $\dot H_a^0$ is dense in $\dot H_a^{s-1}$,
there is a sequence $F_n\in \dot H_a^0$ converging to $A_a w$
in $\dot H_a^{s-1}$.  Setting
$u_n := (A_a)^{-1} F_n$, elliptic regularity implies that $u_n $
converges to $w$ in $\tH{1+s}$.  
Clearly $u_n $ is in $D(A)$,
i.e., $D(A)\cap \tH{1+s}$ is dense in 
$\tH{1+s}$ as claimed. 

Suppose that $u\in D(A)\cap \tH{1+s}$.
We first show that 
\beq
\|A^{ (1+s)/2} u\| \le C \|u \|_{H^{1+s}}.
\label{first}
\eeq
For $v\in D(A^*)$ and $\delta := (1-s)/2\in [0,1/2)$,
$$(A^{(1+s)/2} u,v) = ( A^{-\delta} A u,v) =(Au,(A^*)^{-\delta}v).
$$
Since $u\in D(A)$ and $(A^*)^{-\delta}v\in D(A^*)\subset \mathbb V$, 
$$\bal
|( A u, (A^*)^{-\delta} v)|&=|\langle A_a u ,  (A^*)^{-\delta} v\rangle|
\le \|A_au\|_{\dot H_a^{-1+s}}
\|\Ats^{-\delta} v\|_{\dH {1-s}}\\
&\preceq \|u\|_{H^{1+s}}
\|\Ats^{-\delta} v\|_{\dH {1-s}}
\eal
$$ 
where we also used  \eqref{dual} and \eqref{a:Ta}.
Now, Theorem~\ref{cont0_1_2} ensures that
$$\|\Ats^{-\delta} v\|_{\dH {1-s}}\preceq \|\Ats^\delta \Ats^{-\delta} v\|= \|v\|.$$
Combining the above inequalities shows that   \eqref{first} holds for $u\in D(A)\cap 
\tH{1+s}$.
The inclusion 
$\tH{1+s} \subset D(A^{(1+s)/2})$ and  
\eqref{first} for  \modif{$u\in \tH{1+s} $} hold by density.  
\end{proof}

Let $T_{S,a}$ be defined similarly to $T_a$ but using the form $S(\cdot,\cdot)$.
The final result in this section shows that under suitable assumptions,
Lemma~\ref{l:characterization_discrete} holds for $s=1$.  This is
a discrete Kato Square Root Theorem.   Its proof was motivated by the
proof of the Kato Square Root Theorem given in \cite{Agranovich}.

\begin{theorem} [Discrete Kato Square Root Theorem]\label{t:discrete_kato}
  Suppose that  \eqref{coer} and 
  \eqref{bdd} hold.
  Assume furthermore that for some   $s \in (0,1/2)$,   
  $T_a$, $T_a^*$ and $T_{S,a}$  are isomorphisms from $\dot H_a^{-1+s}$
  into $\tH{1+s}$.
   Then,  there are positive constants $c,C$ independent of $h$ such that
$$ 
c \|v_h\|_{\mbV}\le \|A_h^{1/2} v_h\|\le C   \|v_h\|_{\mbV}, \Forall v_h\in \mbV_h.
$$
The analogous inequalities hold with $A_h^*$ replacing $A_h$ above.
\end{theorem}

\begin{proof}  
In this proof $C$ denotes a generic constant independent of $h$. 
The assumption on $T_{S,a}$ implies that for
  $t\in (0,s]$, $\tH{1+t}=\dH {1+t}$,
with equivalent norms \cite{bp-fractional}.
  
We first observe that it suffices to show that
\beq
\| A_h^{1/2} u_h \|\le C \|u_h\|_\mbV,\Forall u_h\in \mbV_h
\label{toshow1/2}
\eeq
along with the analogous inequality involving $A_h^*$.
Indeed, if \eqref{toshow1/2} holds then by \eqref{coer},
$$\bal
c_0 \|u_h \|_\mbV^2 &\le |(A_h u_h,u_h)| \le
\|A_h^{1/2}u_h\| \|(A_h^*)^{1/2} u_h \| \\
&\le C \|(A_h^*)^{1/2} u_h \| \| u_h \|_\mbV
\eal
$$
and hence
$$ c_0 \|u_h \|_\mbV\le C \|(A_h^*)^{1/2} u_h \|.$$
The proof of the lower bound involving $A_h^{1/2}$ is similar.

Applying  Lemma~\ref{l:characterization_discrete} gives, for $u_h\in \mbV_h$,
\beq
\bal \|A_h^{(1+s)/2} u_h\|&=\sup_{\theta\in \mbV_h }
\frac {(A_h^{(1+s)/2}u_h,(A_h^*)^{(1-s)/2}\theta)}
      {\|(A_h^*)^{(1-s)/2}\theta\|}\le C\sup_{\theta\in \mbV_h }
\frac {(A u_h,\theta)}
      {\|\theta\|_{\dH {1-s}}}\\
        & \le C\|A_au_h\|_{\dH{s-1}} \le C \|u_h \|_{\dH {1+s}}
\eal
       \label{ahbound}
\eeq
where we used the assumption on $T_a = (A_a)^{-1}$.        

Let $\tph$ denote the $S$-elliptic projection onto $\mbV_h$, i.e.,
for $v\in \mbV$, $w_h=\tph v\in \mbV_h$   solves
$$S(w_h,\theta_h)=S(v,\theta_h),\Forall \theta_h\in \mbV_h.$$
It is a consequence 
of the isomorphism assumption on $T_{S,a}$ that
$\tph$ is a uniformly (independent of $h$) bounded operator on
$\tH{1+s}$
(see, e.g., \cite{MR1052086}).
Thus, recalling the eigenvalue decomposition \eqref{e:dot_intermediate}, it holds for $u_h\in \mbV_h$,
$$\bal \|u_h\|_{\dH {1-s}} &= \sup_{\phi\in \dH {1+s} }
\frac {S(u_h,\phi)} {\|\phi\|_{\dH {1+s}}} =
\sup_{\phi\in \dH {1+s} }
\frac {S(u_h,\tph \phi)} {\|\tph\phi\|_{\dH {1+s}}} \frac
      {\|\tph\phi\|_{\dH {1+s}}}
      {\|\phi\|_{\dH {1+s}}}\\
      &\le C\sup_{\phi_h\in \mbV_h }
      \frac {S(u_h,\phi_h)} {\|\phi_h\|_{\dH {1+s}}}.
      \eal
$$
Similarly, for all $u_h\in \mbV_h$,
$$ \|u_h\|_{\dH {1+s}} = \sup_{\phi_h\in \mbV_h }
\frac {S(u_h,\phi_h)} {\|\phi_h\|_{\dH {1-s}}}.
$$
Thus, the characterization \eqref{sh1} gives
$$\|u_h\|_{\dH {1+s}}\le C \sup_{\phi_h\in \mbV_h }
\frac {S(u_h,\phi_h)} {\|S_h^{(1-s)/2}\phi_h\|}
= C\|S_h^{(1+s)/2}u_h\|.$$
Combining this with \eqref{ahbound} gives
$$\|A_h^{(1+s)/2} u_h\|\le C \|S_h^{(1+s)/2}u_h\|, \Forall u_h\in \mbV_h.$$
 Interpolating this result with the trivial inequality
$$\|A_h^{0} u_h\|\le \|S_h^0u_h\|,\Forall u_h\in \mbV_h,$$
gives 
$$\|A_h^{1/2} u_h\|\le C \|S_h^{1/2}u_h\|=C\|u_h\|_\mbV,\Forall u_h\in \mbV_h$$
and verifies \eqref{toshow1/2}. The proof for the analogous inequality
  involving $A_h^*$ is similar.  
  \end{proof}

\section{Exponentially Convergent sinc Quadrature.}\label{s:SINC}

Theorem~\ref{FEinterror} and Remark~\ref{r:conv_s12} provide estimates for the errors $\| (A^{-\beta} - A_h^{-\beta}\pi_h) f\|_{\dH{2s}}$.
We now apply an exponentially convergent sinc quadrature (see,  for
example, \cite{lundbowers}) to approximate the integral in 
 \eqref{A-betah}.
Since the argument below does not require the operator
$A_h^{-\beta}\pi_h$ to be discrete, the analysis includes and
focuses on the case of
$A^{-\beta}$.

The change of variable and sinc quadrature approximations for $A$ 
are thus
$$
A^{-\beta}= \frac{\sin(\pi \beta)}{\pi} \int_{-\infty}^\infty e^{(1-\beta)y} (e^yI+A)^{-1} ~ dy
$$
and
$$
\cQ^{-\beta}_k(A) := \frac{k\sin(\pi \beta)}{\pi}      \sum_{\ell=-N}^N e^{(1-\beta) y_\ell} (e^{y_\ell}I+  A)^{-1} .
$$
Here, for any positive integer $N$, $y_\ell:=\ell k$ and  $k:=1/\sqrt N
$.

\def\bD{{\overline {D}}}
To estimate the quadrature error, we start by defining $D_{\pi/2} := \{ z\in
\CC,\ |\Im(z)|<\pi/2\}$ and denote by $\bD_{\pi/2} $ its closure. 
Here $\omega = \arctan(\eta)$ and $\eta$ is the index of the sesquilinear form $A(.,.)$, see \eqref{index}.
For any functions $u,v \in L^2(\Omega)$ and $z\in \CC$  with $-e^z$ not in the spectrum of
$A$, we define
$$
f(z;u,v):=  e^{(1-\beta)z}\left( (e^z I+A)^{-1}u,v \right).
$$
Note that $\Re(e^z)$ is non negative for $z\in \bD_{\pi/2} $ and hence 
Theorem~\ref{resbound} and Remark~\ref{cz2} imply that $f(z,u,v)$ is
well defined for $z\in \bD_{\pi/2} $. 

We apply the classical analysis for these types of quadrature approximations given in \cite{lundbowers} with a particular attention on deriving estimates uniform in $u,v \in L^2(\Omega)$.
Using the resolvent estimate (Theorem~\ref{resbound}) when  $\Re(z) > 0$
and Remark~\ref{cz2} when $\Re(z) \leq 0$ gives
$$ \|(e^zI+A)^{-1}\|\le \left \{\bal |e^{-z}|/\sin(\pi/2-\omega)&:\quad  \hbox{ for } z\in
\bD_{\pi/2},\ \Re(z)>0,\\
\frac2{c_0} &:\quad  \hbox{ for } z\in
\bD_{\pi/2},\ \Re(z)\le 0\eal\right .$$
\modif{It is a consequence of the above inequality  that $f(t,u,v)$ is an analytic
function of $t$ for 
$t\in D_{\pi/2}$.  Moreover,
for $z \in \bD_{\pi/2} $:}
\begin{equation}
\begin{split}
&|e^{(1-\beta)z} \left( (e^z I+A)^{-1}u,v \right)| \\
& \qquad \qquad \le \|u\| \|v\| \left\lbrace 
\begin{array}{ll}
 e^{-\beta\Re(z)}/\sin(\pi/2-\omega)&:\qquad \hbox{ for } \Re(z) >0,\\
 \frac{2}{c_0} e^{(1-\beta) \Re(z)}&:\qquad \hbox{ for } \Re(z) \le 0.
\end{array} \right . 
\end{split}
\label{dbinf}
\end{equation}
This implies that 
\begin{equation}
\begin{split}
\int _{-\infty}^\infty &( |f(y-i\pi/2;u,v) | +  |f(y+i\pi/2;u,v)
|) \, dy  \\& \le \left(2(\beta\sin(\pi/2-\omega))^{-1} + 4((1-\beta)c_0)^{-1}\right):=N(D_{\pi/2})\|u\|\|v\|.
\end{split}
\label{nbinf}
\end{equation}
In addition, applying \eqref{dbinf} gives
$$
\int_{-\pi/2}^{\pi/2} | f(t+iy,u,v)|dy \le  C,\Forall t\in \RR.
$$
\modif{Hence, we can}  apply Theorem 2.20 of \cite{lundbowers}  to conclude  
that for $k>0$,
\beq
  \bigg| 
\int_{-\infty}^\infty f(y;u,v)   \ dy-
k\sum_{\ell=-\infty}^\infty f(\ell k;u,v) 
\bigg|
 \le
\frac{N(D_{\pi/2};u,v)}{2\sinh(\pi^2/(4k))} e^{-\pi^2/(4k)}.
\label{220res}
\eeq
\modif{
This leads to  the following result for the sinc quadrature error.
}

\begin{theorem}[sinc Quadrature Error]   \label{l:exp}
For $N>0$ and $k:=1/\sqrt{N}$, let $\cQ^{-\beta}_k(\cdot)$ be defined by 
\eqref{quadh}.
Then for $s\in [0,1)$, there exists a constant $C(\beta)$ independent of $k$ such that
\begin{equation}
\begin{split}
\|A^{-\beta} - \cQ^{-\beta}_k(A) \|_{\dH{s} \to \dH{s}} \le
C(\beta) &
\bigg[
\frac{N(D_{\pi/2})}{2\sinh(\pi^2/(4k))}e^{-\pi^2/(4k)} \\
&+ \frac 1 \beta e^{-\beta/k} +
\frac{2}{(1-\beta)c_0} e^{-(1-\beta)/k}\bigg].
\end{split}
\label{aquad}
\end{equation}
Similarly, there exists a constant $C(\beta)$ independent of $k$ and $h$ such that
\begin{equation}
\begin{split}
 \|(A_h^{-\beta} - \cQ^{-\beta}_k(A_h))\pi_h \|_{\dH{s} \to \dH{s}}
 \le C(\beta) &
 \bigg[
\frac{N(D_{\pi/2})}{2\sinh(\pi^2/(4k))}e^{-\pi^2/(4k)}\\
& + \frac 1 \beta e^{-\beta/k} +
\frac{2}{(1-\beta)c_0} e^{-(1-\beta)/k}
\bigg].
\end{split}
\label{ahquad}
\end{equation}
\end{theorem}

\begin{proof}
Both estimates when $s=0$ directly follow from \eqref{220res}
and the estimates
\begin{align*}
&k  \sum_{\ell=-N-1}^{-\infty} |f(\ell k; u,v) | 
\le \frac 1 { \beta} e^{-\beta/k} \|u\|\|v\|,  \\
&k  \sum_{\ell=N+1}^\infty |f(\ell k;u,v)|
\le\frac 2{(1-\beta)c_0}
e^{-(1-\beta)/k} \|u\|\|v\|,
\end{align*}
(which are direct consequences of \eqref{dbinf}).

For the case $s\in (0,1)$, we first consider  \eqref{aquad}.
Using Theorem~\ref{cont0_1_2} and the commutativity of  $A$ and  $(e^{y}I+A)^{-1}$, we have
$$
 \|A^{-\beta} - \cQ^{-\beta}_k(A) \|_{\dH s \to \dH{s}} \preceq 
 \|A^{s/2}(A^{-\beta} - \cQ^{-\beta}_k(A))A^{-s/2} \|.
$$ 
Applying \eqref{anegcom} shows that the right hand side above equals 
$$ \|A^{-\beta} - \cQ^{-\beta}_k(A) \|$$
and so the desired estimate follows again from the $s=0$ case.

For \eqref{ahquad} when $s>0$, we apply 
Lemma~\ref{l:characterization_discrete} to see that
\begin{align*}
 \|(A_h^{-\beta} - \cQ^{-\beta}_k(A_h))\pi_h \|_{\dH{s} \to \dH{s}}
 &\le \|A_h^{s/2}(A_h^{-\beta} - \cQ^{-\beta}_k(A_h))A_h^{-s/2} \|
 \|\pi_h\|_{\dH{s} \to \dH{s}}\\
&\preceq 
  \|(A_h^{-\beta} - \cQ^{-\beta}_k(A_h))\pi_h \|.
\end{align*}  
The last inequality followed from commutativity and  
the fact that $\pi_h$ is a bounded operator on
$\dH s$.
The result now follows from the $s=0$ case.
\end{proof}

\begin{remark}[Critical Case $s=1$]
The quadrature error estimate still holds when $s=1$ provided
that the discrete and continuous 
Kato Square Theorems hold.   These results are given by
Remark~\ref{re:KatoT} for the continuous operator $A$ and 
Theorem~\ref{t:discrete_kato} for the discrete operator $A_h$.
\end{remark}

\begin{remark}[Exponential Decay] \label{rem:MN} The error from the three exponentials above 
can   essentially be equalized by setting
\beq
\cQ^{-\beta}_k (A_h):= \frac{k\sin(\pi \beta)}{\pi}      \sum_{\ell=-M}^N
e^{(1-\beta) y_\ell} (e^{y_\ell}I+  A_h)^{-1} 
\label{inforder}
\eeq 
with
$$\pi^2/(2k)\approx 2\beta k M\approx (2-2\beta) kN.$$
Thus, given $k>0$, we set 
$$ M=\bigg\lceil \frac {\pi^2}{4\beta k^2}\bigg\rceil \quad \hbox{ and }\quad
N=\bigg \lceil \frac {\pi^2}{4(1-\beta) k^2}\bigg \rceil  
$$
and get the estimate
$$
 \|(A_h^{-\beta} - \cQ^{-\beta}_k(A_h))\pi_h \|_{\dH{s} \to \dH{s}}
 \le C(\beta)
 \bigg[ \frac 1 {2\beta} + \frac 1 {2(1-\beta)\lambda_0} \bigg] \bigg[  \frac {e^{-\pi^2/(4k)}}{\sinh(\pi^2/(4k))}+
e^{-\pi^2/(2k)} \bigg].
$$
We note that the right hand side above asymptotically behaves like 
$$C(\beta) \bigg[ \frac 1 {2\beta} + \frac 1 {2(1-\beta)\lambda_0} \bigg]  e^{-\pi^2/(2k)}$$
as $k\rightarrow 0$. 
\end{remark}

This, together with the finite element approximation estimates provided in Theorem~\ref{FEinterror} and Remark~\ref{r:critical_s1}, yields the fully discrete convergence estimate stated below.

\begin{corollary}[Fully Discrete Convergence Estimate]\label{c:fully}
Suppose that  \eqref{coer} and
  \eqref{bdd} as well as
  the elliptic regularity Assumption~\ref{ellreg} hold.  
  Given $s\in [0,\frac 12)$, set $\alpha_*:=\frac 1 2 (\alpha+\min(1-2s,\alpha))$ and 
  $\gamma:=\max(s+\alpha^*-\beta,s)$. For $\delta \geq \gamma$, then there exists a constant $C$ independent of $k<1$ and $h$ such that
$$\bal
\| (A^{-\beta  } - \cQ_k^{-\beta}(A_h)\pi_h) f\|_{\dH{2s}} 
\leq C_{\delta,h}  h^{2\alpha_*}\|A^\delta f\| &+ C
e^{-\pi^2/2k}\|f\|_{\dH{2s}},\\
&\quad  \forall f\in
D(A^{\delta})\cap \dH{2s},
\eal
$$
where $C_{\delta,h}$ is given by \eqref{cdelta} 
and $\cQ_h^{-\beta}$ is defined by \eqref{quadh}.

In addition, if the continuous and discrete Kato Square Root Theorems 
hold (see Remark~\ref{re:KatoT} and Theorem~\ref{t:discrete_kato})
then the above estimate also holds for $s=\frac 1 2$.
\end{corollary}  

\section{Numerical Illustrations for the Convection-Diffusion Problem}\label{s:numerical}

In order to illustrate the performance of the proposed algorithm,  we
set $\mbV=H^1_0(\Omega)$ with $\Omega=(0,1)^2$ and for
$b\in \mathbb R$, consider the sesquilinear form
$$
A(u,v):=\int_\Omega (\nabla u \cdot \nabla
\bar v+b(u_x+u_y)\bar v),\Forall u,v\in \mbV.
$$
This form is regular and the corresponding regularly
accretive operator $A$ has domain $H^2(\Omega)\cap \mbV$.

In general, it is difficult to \modif{analytically} compute solutions to $u=A^{-\beta} f$
although it is possible in this case.  Indeed, we
consider the Hermitian form:
$$
\tS(u,v)=\int_\Omega \bigg (\nabla u \cdot \nabla
\bar v+\frac{b^2}2u\bar v\bigg),\Forall u,v\in \mbV.$$
Fix $f\in L^2(\Omega)$ and for $\mu\ge 0$, let $w\in \mbV$ solve
$$\mu(w,\phi)+A(w,\phi)=(f,\phi),\Forall \phi\in \mbV.$$
Putting $v=e^{-b(x+y)/2} w$ and $\phi=e^{-b(x+y)/2} \theta$  (for
$\theta\in \mbV$) in the above
equation and integrating by parts when appropriate,  we see that
$v\in \mbV$ is the solution of
$$\mu(v,\theta)+\tS(v,\theta)=(e^{-b(x+y)/2}f,\theta),\Forall \theta\in
\mbV,$$
i.e.,  
$$e^{-b(x+y)/2} (\mu+A)^{-1}f= (\mu+\tS)^{-1} e^{-b(x+y)/2}f.$$
Substituting this into \eqref{A-beta} shows that
$$A^{-\beta} f = e^{b(x+y)/2} \tS^{-\beta} (e^{-b(x+y)/2}f).$$
Since, for example,  $\sin(\pi x)\sin(2\pi y) $ is an eigenvector  of
$\tS$ with eigenvalue $5\pi^2+b^2/2$, 
$$u:=A^{-\beta} f =
e^{b(x+y)/2} (5\pi^2+b^2/ 2)^{-\beta}\sin(\pi x)\sin(2\pi y)$$
when
\beq
f=e^{b(x+y)/2}\sin(\pi x)\sin(2\pi y).
\label{fanal}
\eeq

The space discretization consists of continuous piecewise bi-linear finite element subordinate to successive quadrilateral refinements of $\Omega$.
Figure \ref{f:convergence_advection} provides the behavior of the errors
$e_h^k:=\| (A^{-\beta} -\cQ^{-\beta}_k(A_h)\pi_h ) f \|$
when the advection coefficient is given by $b=1$ and $f$ is given by
\eqref{fanal} . For a fixed number of quadrature points $e_h^k \sim h^2$ while for a fixed spatial resolution $e_h^k$ is exponential decaying. This is in agreement with the estimate provided in Corollary~\ref{c:fully}.

\begin{figure}
\includegraphics[width=0.45\textwidth]{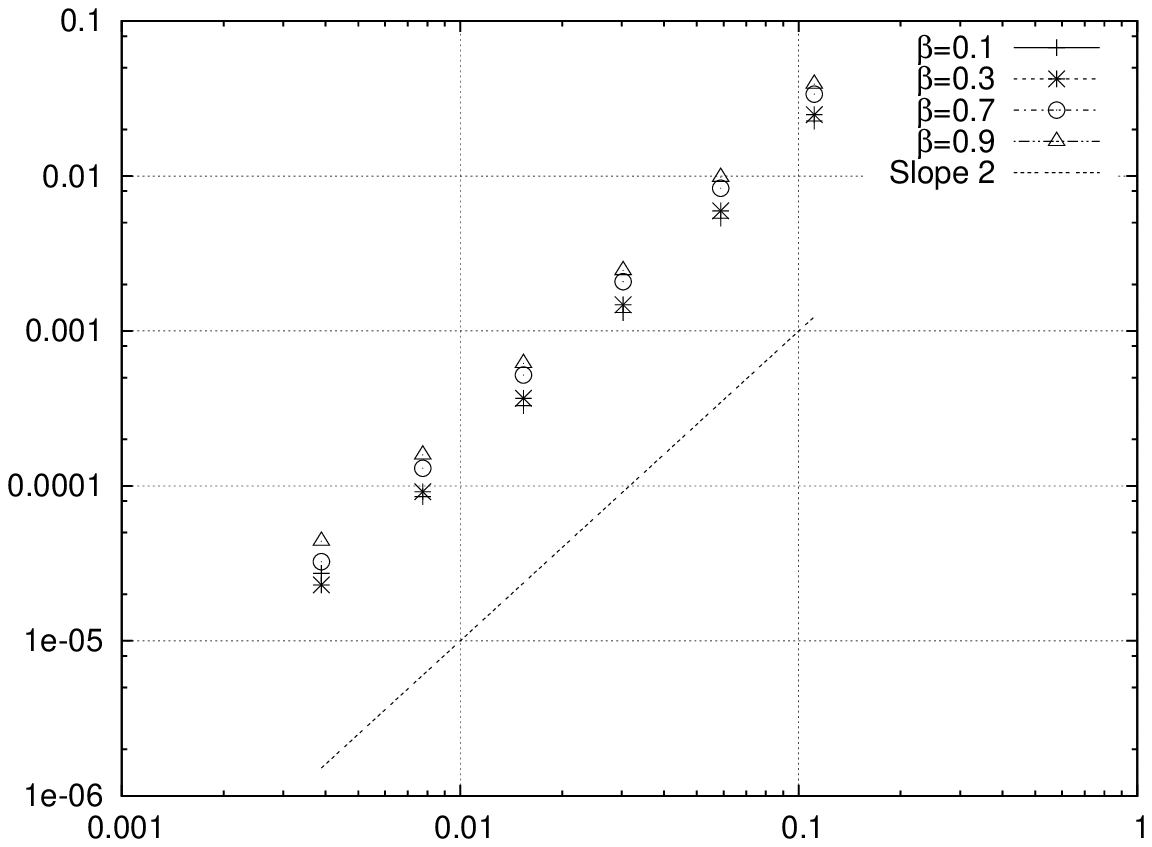}
\includegraphics[width=0.45\textwidth]{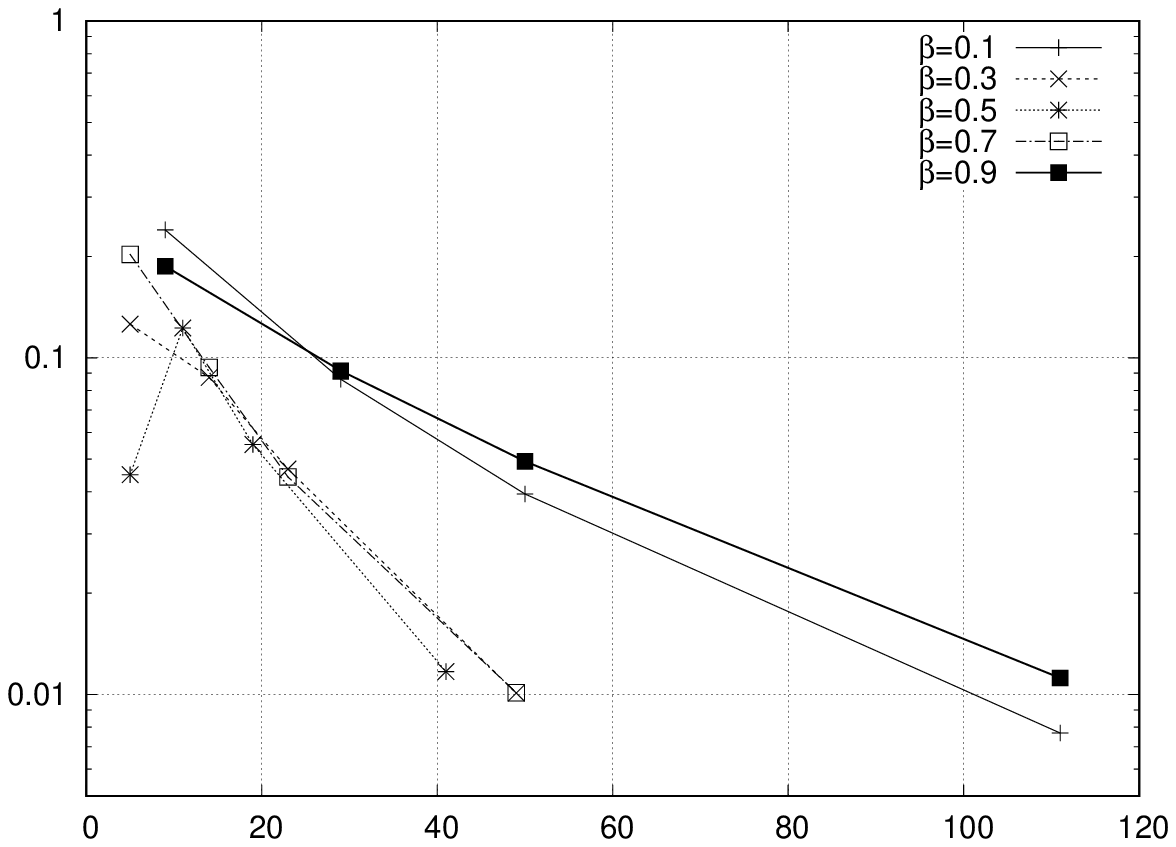}
\caption{(Left) Decay of $e_h^k$ versus the uniform mesh size $h$ \modif{in a log-log plot}. Second order rate of convergence is observed for all values of $\beta$. The number of quadrature points is taken large enough not to interfere with the spatial discretization error. (Right) Exponential decay of $e_h^k$ as a function \modif{of the square root} of number of points $M+N+1$ \modif{in a semi-log plot} (see Remark~\ref{rem:MN} for the definition of $k$). The spatial discretization is fixed and consists in 10 uniform refinements of $\Omega$.}\label{f:convergence_advection}
\end{figure}

We now set $b=10$, $f\equiv 1$ and study the boundary layer inherent to convection-diffusion problems.
For this, we consider 8 successive quadrilateral refinements of $\Omega$ for the space discretization.
In particular, the corresponding mesh size $h:=2^{-8}$ is fine enough for the Galerkin representation not to require any stabilization. 
The value of the approximations $A_h^{-\beta}1$ for $\beta=0.1,0.3,0.5,0.7,0.9$ over the segment joining the points $(0,0)$ and $(1,1)$ are plotted in Figure \ref{f:con_diff} together with the graphs of $5A_h^{-\beta}1$ for $\beta=0.1$ and $\beta=0.9$.
The results indicate that the width of the boundary layer for convection-diffusion problems remains proportional to the ratio diffusion / convection and is therefore independent of $\beta$.
However, its intensity decreases with increasing $\beta$.

\begin{figure}\label{f:con_diff}
\begin{tabular}{cc}
\includegraphics[width=0.45\textwidth]{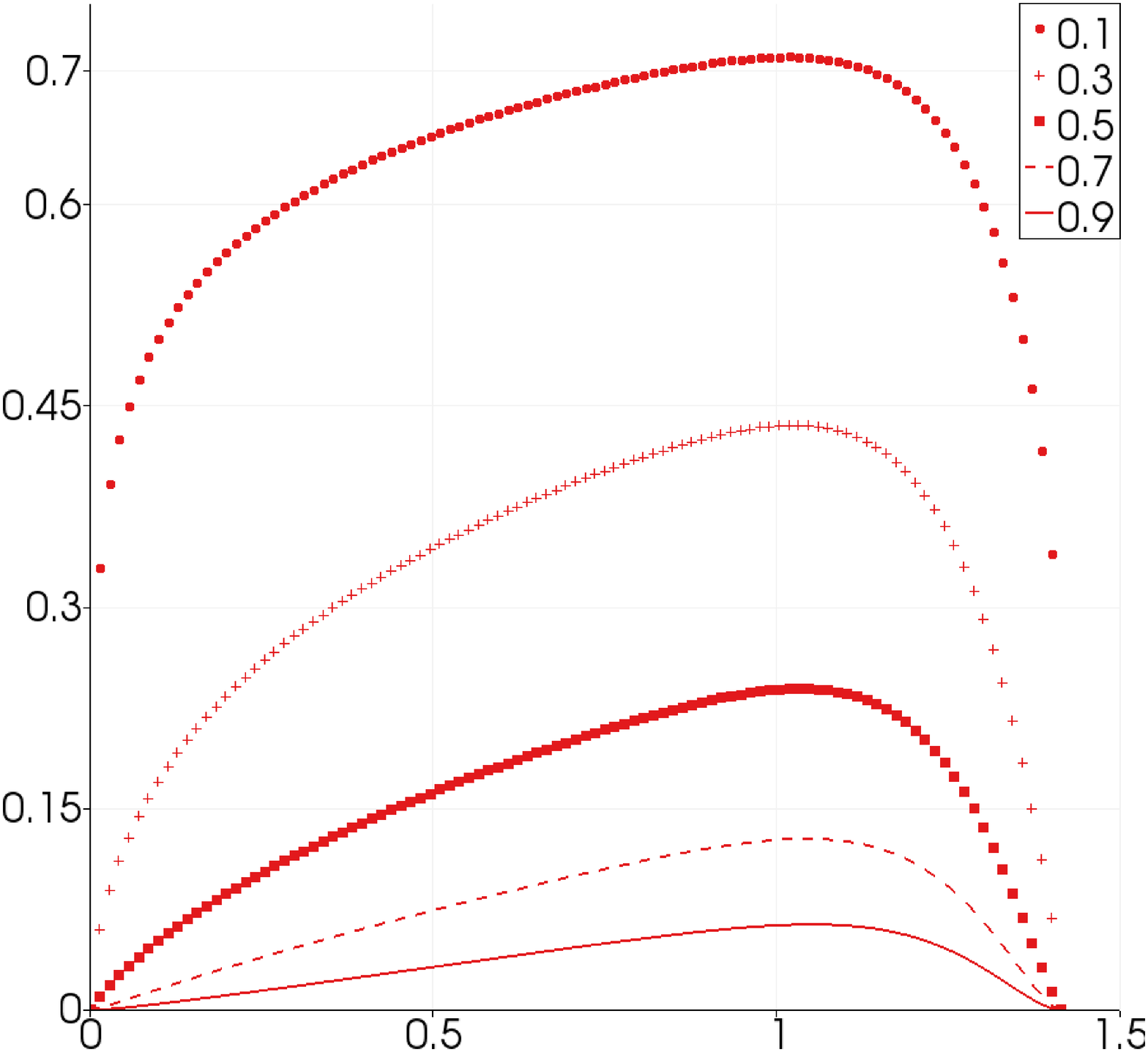}&
\includegraphics[width=0.4\textwidth]{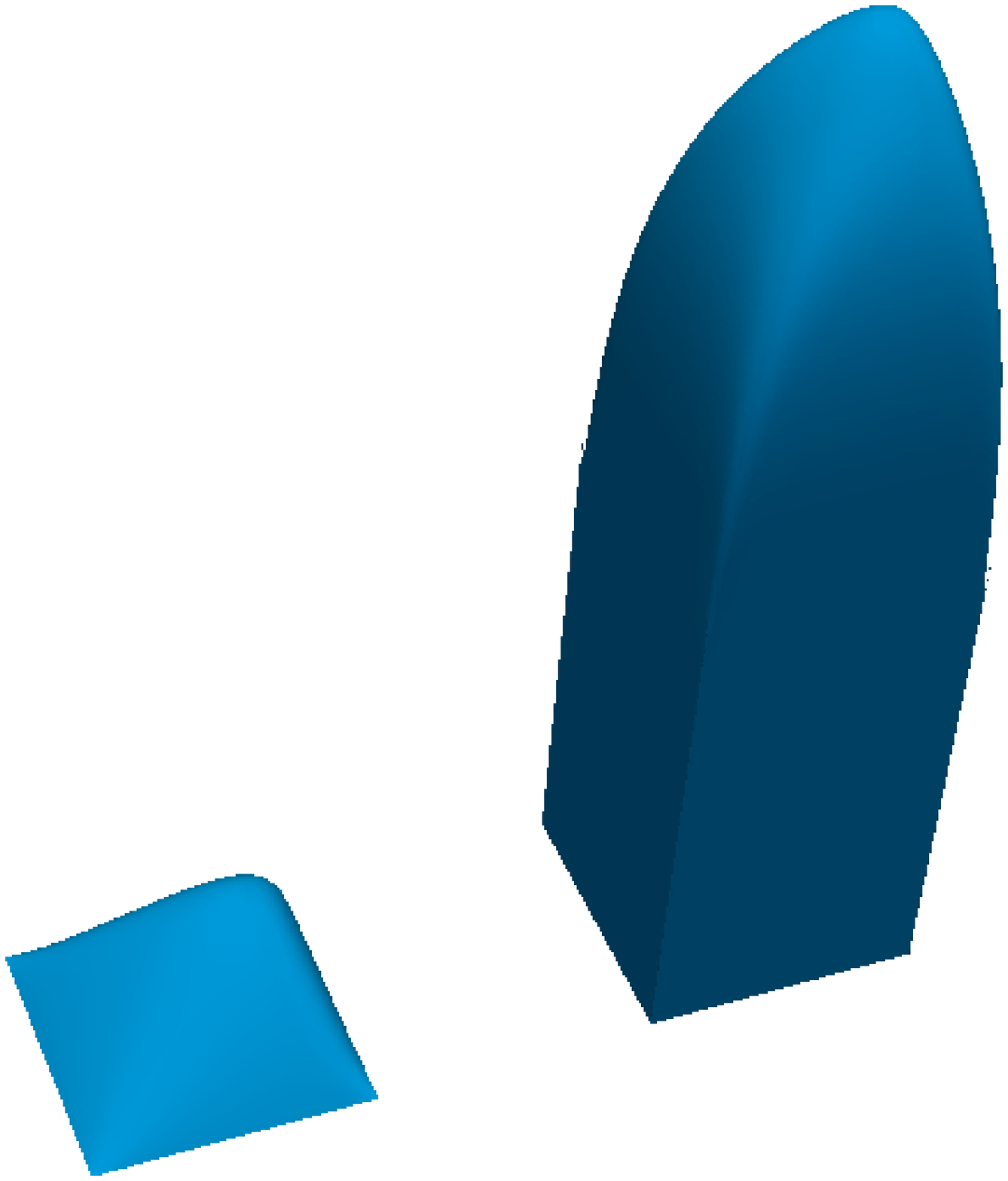}
\end{tabular}
\caption{Approximations of $A^{-\beta}1$ on a subdivision of the unit square using $4^8$ quadrilaterals and $401$ quadrature points.
The width of the boundary layer for convection-diffusion problems appears independent of $\beta$ while its intensity decreases with increasing $\beta$.
(Left) Plots over the segment joining the $(0,0)$ and $(1,1)$ for $\beta=0.1,0.3,0.5,0.7,0.9$.
(Right) Approximations scaled by a factor $5$ for $\beta=0.9$ and $\beta=0.1$.}
\end{figure}

\section*{Acknowledgment}
The first author was
partially  supported by
  the National Science Foundation through Grant DMS-1254618 while the 
second  was partially supported by
  the National Science Foundation through Grant DMS-1216551.
The numerical experiments are performed using the \emph{deal.ii} library \cite{BHK:07} and \emph{paraview} \cite{paraview} is used for the visualization.

\bibliographystyle{plain}

\def\cprime{$'$}

\end{document}